\documentclass[a4paper,reqno,10pt]{amsart}
\usepackage[left=2.8cm,right=2.8cm,top=2.8cm,bottom=2.8cm]{geometry}
\usepackage[T1]{fontenc}
\usepackage{indentfirst} 
\usepackage{amsmath,amsfonts,amssymb,,mathrsfs}

\usepackage{bbm}
\usepackage[colorlinks=true]{hyperref}
\usepackage{enumerate}
\usepackage{graphicx}
\usepackage{xcolor}
\usepackage[all]{xy}
\usepackage{authblk}     
\usepackage{footmisc}    
\usepackage{abstract}    
\usepackage{titling}     
\usepackage{fancyhdr}

\newcommand{\restr}{%
  \,\raisebox{-.127ex}{\reflectbox{\rotatebox[origin=br]{-90}{$\lnot$}}}\,%
}

\theoremstyle{definition}
\newtheorem{definition}{Definition}[section]
\theoremstyle{remark}
\newtheorem{remark}[definition]{Remark}
\theoremstyle{theorem}
\newtheorem{theorem}[definition]{Theorem}
\theoremstyle{theorem}
\newtheorem{lemma}[definition]{Lemma}
\theoremstyle{remark}

\theoremstyle{theorem}
\newtheorem{corollary}[definition]{Corollary}
\theoremstyle{theorem}
\newtheorem{proposition}[definition]{Proposition}
\theoremstyle{theorem}

\numberwithin{equation}{section}
\allowdisplaybreaks

\def\lambdant#1{\mathchoice
    {\XXint\displaystyle\textstyle{#1}}%
    {\XXint\textstyle\scriptstyle{#1}}%
    {\XXint\scriptstyle\scriptscriptstyle{#1}}%
    {\XXint\scriptscriptstyle\scriptscriptstyle{#1}}%
    \!\int}
\def\XXint#1#2#3{{\setbox0=\hbox{$#1{#2#3}{\int}$}
      \vcenter{\hbox{$#2#3$}}\kern-.5\wd0}}

\def\mint{\lambdant-}

\definecolor{ao}{rgb}{0.0, 0.5, 0.0}

\DeclareMathOperator*{\esup}{ess\sup}

\title{\textbf{Rectifiability of a class of integral geometric measures and applications}}

\newcommand{\runningtitle}{Rectifiability of integral geometric measures}

\author{Emanuele Tasso}
\affil[1]{TU Wien, Institute of Analysis and Scientific Computing}
\affil{Wiedner Hauptstrasse 8-10, 1040 Vienna (Austria)}

\begin{document}
\maketitle

\section*{Acknowledgements}
This paper has been supported by the Austrian Science Fund (\textbf{FWF}) projects \textbf{Y1292} and \textbf{F65}. 

\vspace{7mm}
\begin{abstract}
We resolve a long-standing open problem posed by Federer concerning the rectifiability of the integral geometric measure with exponent $p \in (1,\infty]$, thereby settling a question that has persisted since its formulation in \cite[3.3.16]{fed1} and further discussed in \cite{mat2, mat4}. As an application, we establish two novel results related to Vitushkin’s conjecture. First, in a multi-scale setting, we provide an affirmative answer to Vitushkin’s conjecture for sets with finite integral geometric measure, within regimes of Favard length behavior at small scales that have not been previously addressed. Second, in a single-scale framework, we extend the Besicovitch-Federer projection theorem beyond the classical $\sigma$-finite setting, namely for planar sets intersecting a typical line in finitely many points. 
\end{abstract}


\vspace{3mm}

\textbf{Keywords:} Integral geometric measure, Besicovitch-Federer's projection theorem, Favard length, Vitushkin's conjecture, unrectifiable set, disintegration

\section{Introduction}
The origins of integral geometric measure can be traced back to classical problems in geometric probability. In the 18\textsuperscript{th} century, Buffon's needle problem introduced the idea of using probabilistic methods to evaluate geometric quantities like length and area. Building upon this foundation, Crofton later formulated what is now known as the \emph{Crofton formula}, which relates the length of a planar curve to the expected number of its intersections with random lines.

These ideas were rigorously developed in the 20\textsuperscript{th} century by Santaló, who formalized the field of integral geometry by introducing a systematic framework based on group-invariant measures over spaces of geometric configurations, such as the Grassmannian of subspaces. 

Within this framework, the $m$-dimensional integral geometric measure of a Borel set $B \subset \mathbb{R}^n$ is defined by
\begin{equation}
\label{e:intgeodirectdef}
\mathcal{I}^m_1(B) := \int_{\mathrm{Gr}(n,m)} \left( \int_{\mathbb{R}^n} \#\left( B \cap \pi_V^{-1}(x) \right) \, d\mathcal{H}^{m}(x) \right) d\gamma_{n,m}(V),
\end{equation}
where $\mathrm{Gr}(n,m)$ denotes the Grassmannian of $m$-dimensional linear subspaces of $\mathbb{R}^n$, $\pi_V^{-1}(x)$ is the affine fiber over $x \in V$, $\mathcal{H}^m$ is the $m$-dimensional Hausdorff measure, and $\gamma_{n,m}$ is the natural rotation-invariant probability measure on the Grassmannian. The symbol $\#$ denotes the number of points of the intersection $B \cap \pi_V^{-1}(x)$

One of the central results in integral geometry, viewed as a generalization of Crofton's classical formula, is that, for sufficiently regular sets, such as $m$-rectifiable sets, the integral geometric measure $\mathcal{I}^m_1$ coincides with the $m$-dimensional Hausdorff measure $\mathcal{H}^m$, up to a constant depending on $n$ and $m$. This identification was a key achievement in integral geometry, as it allows geometric quantities to be understood in terms of intersection probabilities, and conversely, enables the derivation of probabilistic laws from geometric structure.

In modern geometric measure theory, Federer \cite{fed1} provided an equivalent definition of the integral geometric measure using the concept of the \emph{Favard length}. For a Borel set $B \subset \mathbb{R}^n$, the Favard length is defined as
\[
\mathrm{Fav}(B) := \int_{\mathrm{Gr}(n,m)} \mathcal{H}^m(\pi_V(B)) \, d\gamma_{n,m}(V),
\]
The measure $\mathcal{I}^m_1$ can then be obtained via the classical Carathéodory procedure by means of the Favard length and the family of coverings  consisting of all Borel sets (see \cite[Subsection 2.10.5]{fed1}). 

Notice that the behavior of the integral geometric measure is intimately tied to the geometric structure of sets. Indeed, the Besicovitch–Federer projection theorem characterizes purely unrectifiable sets with finite Hausdorff measure as those having zero Favard length. This characterization allows one to completely determine the behavior of $\mathcal{I}^m_1$ on sets with finite $m$-dimensional Hausdorff measure (see \cite[Theorem 3.2.26]{fed1}).

In order to extend this understanding to all sets, Federer introduced a natural generalization involving uniformly convex constraints, parameterized by an exponent $p \in (1, \infty]$,
\begin{equation}
\label{e:favp}
\mathrm{Fav}_p(B) := \left( \int_{\mathrm{Gr}(n,m)} \mathcal{H}^m(\pi_V(B))^p \, d\gamma_{n,m}(V) \right)^{\frac{1}{p}},
\end{equation}
which gives rise, via Caratheodory's contruction, to a corresponding family of integral geometric measures $\mathcal{I}^m_p$ (see \cite[Subsection 2.10.5]{fed1}). This generalization was partly motivated by a structure result of Mickle \cite{mic}, which implies that
\[
\mathcal{I}^m_\infty(E) < \infty \ \Rightarrow \ \mathcal{I}^m_p(E) \simeq \mathcal{I}^m_\infty(E), \ \ \text{for every $p \in [1,\infty]$},
\]
up to a renormalization constant depending only on the dimension and the exponent.

\vspace{5mm}

 This leads Federer in \cite[Subsection 3.3.16]{fed1} to formulate the fundamental question (see also the more recent surveys \cite[Problem 9]{mat5} and \cite[Subsection 5.5]{mat4}):

 \vspace{1mm}
\[
\textbf{Federer (Q.1)}: \emph{\text{Do the measures $\mathcal{I}^m_{p_1}$ and $\mathcal{I}^m_{p_2}$ for $p_1 \neq p_2$ coincide up to a constant multiple?}}
\]

\vspace{4mm}
The endpoint cases $p = 1$ and $p = \infty$ were resolved respectively in the negative and positive directions by Mattila~\cite{mat2} and Mickle~\cite{mic}. 

The result in~\cite{mic} (see also \cite{nem}) implies that, starting from the general structure theorem in~\cite[Theorem 3.3.12]{fed1}, which encompasses the celebrated Besicovitch–Federer structure theorem, and with additional work, the assumption $\mathcal{I}^m_\infty(E) < \infty$ implies $\mathcal{I}^m_\infty(E \setminus R)=0$ for some countably $m$-rectifiable set $R \subset \mathbb{R}^n$ (see also~\cite[Theorem 3.3.14]{fed1}). However, this argument appears to be specific to the case $p = \infty$, as discussed in more detail below. 

In the opposite direction, an important negative result is due to Mattila~\cite{mat2}, who constructed a compact set $K \subset \mathbb{R}^2$ such that
\[
0 < \mathcal{I}^1_1(K) < \mathcal{I}^1_p(K) = \infty \quad \text{for every } p \in (1,\infty].
\]
The construction in~\cite{mat2} is based on an idea of Talagrand~\cite{tal}: by iterating a simple geometric operation on an initial parallelogram, one produces (in the limit) a purely unrectifiable set with the desired separation among the $\mathcal{I}^1_p$-measures. 

Beyond these two endpoint cases, the question for intermediate exponents $p \in (1,\infty)$ has remained open.

\subsection{Our solution to question (Q.1)} In this paper, we answer Question~(Q.1) affirmatively for all $p > 1$ by proving the following rectifiability result:

\begin{theorem}
\label{t:fedpronew}
    Let $E \subset \mathbb{R}^n$ be a Borel set with finite $\mathcal{I}^m_p$-measure. If $p >1$ then 
    \begin{equation}
    \label{e:fedpronew}
    \mathcal{I}^m_p(E \setminus R)=0, 
    \end{equation}
    for some countably $m$-rectifiable Borel set $R \subset \mathbb{R}^n$
\end{theorem}

This result immediately implies that the measures $\mathcal{I}^m_{p_1}$ and $\mathcal{I}^m_{p_2}$ are equivalent provided $1 < p_1 \leq p_2 \leq \infty$, thereby resolving Federer's problem. Indeed, the vanishing of $\mathcal{I}^m_p$ is equivalent to the vanishing of the Favard length for any exponent $p \in [1,\infty]$ leading to the coincidence of all integral geometric measures on the class of negligible sets. In addition, as a consequence of the Coarea formula for Lipschitz maps on rectifiable sets (see, e.g., \cite[Theorem 3.3.13]{fed1}), Crofton-type formulas hold for the entire range of exponents:
\begin{equation}
\label{e:compatibility}
\mathcal{I}^m_p(R) = c_{n,m,p} \, \mathcal{H}^m(R), \quad \text{for some constant } c_{n,m,p} > 0 \text{ and every } p \in [1,\infty].
\end{equation}
Theorem~\ref{t:fedpronew} thus completes the desired identification by establishing the behavior of $\mathcal{I}^m_p$ on all sets with finite measure. It is also worth noting that our result does not imply that formula \eqref{e:compatibility} extends to all sets. Indeed, it is possible to construct sets with zero Favard length but non-$\sigma$-finite corresponding Hausdorff measure (see, e.g., \cite[Corollary 2]{mat3}).

We finally note that Theorem~\ref{t:fedpronew} holds under a more general assumption than the one stated above. Specifically, the result remains valid when the $L^p$-norm in the definition of $\mathcal{I}^m_p$ is replaced by the Luxemburg norm of any Orlicz function $\Phi$ with superlinear growth at infinity. This generalization not only extends Federer’s identification problem to a broader functional setting, but also proves particularly meaningful in light of our application to \emph{Vitushkin’s conjecture}. We refer the reader to Subsection~\ref{ss:vitushkin} for a detailed discussion of this connection.

\vspace{3mm}
\textbf{Challenges in the proof of Theorem~\ref{t:fedpronew}} The true difficulty in understanding Question (Q.1) lies in the case of sets with positive Favard length but non $\sigma$-finite $\mathcal{H}^m$ measure. In fact, as observed in \cite[Subsection 3.3.16]{fed1}, any counterexample to (Q.1) must involve a compact set $E \subset \mathbb{R}^n$ with $\mathrm{Fav}(E) > 0$, yet whose integral geometric measure is entirely diffuse with respect to $\mathcal{H}^m$. That is,
\[
\text{Fav}(E) > 0 \quad \text{and} \quad \Theta^m(\mathcal{I}^m_p \restr E, x) = 0 \quad \text{for all } x \in E,
\]
where $\Theta^m(\mu,x)$ denotes the $m$-dimensional Hausdorff density of the measure $\mu$ evaluated at $x$.

\vspace{4mm}

To clarify the difficulties involved in proving Theorem~\ref{t:fedpronew}, we begin by recalling a consequence of Besicovitch’s key three alternatives lemma (see, e.g. \cite[Theorem 3.3.4]{fed1}). Given a finite measure $\mu$ on $\mathbb{R}^n$, this result allows one to decompose a purely $(\mu,m)$-unrectifiable $F \subset \mathbb{R}^n$ (see Definition \ref{d:unrect}) into three pairwise disjoint subsets,
\[
F = F_1(V) \cup F_2(V) \cup F_3(V),
\]
in a way that depends on the $m$-plane $V \in \text{Gr}(n,m)$ and satisfies, for almost every $V$:
\begin{align}
\label{e:(1)}
    & \ \ \ \ \ \mu(F_1(V)) = 0 \\
    \label{e:(2)}
    &\mathcal{H}^m\big(\pi_V(F_2(V))\big) = 0 \\
    \label{e:(3)}
      x \in F_3(V) \ \Rightarrow \ &x \in \text{Closure} \, \{z \ | \ z-x \in V^\perp  \cap F\}.
\end{align}

Now recall that, to achieve \eqref{e:fedpronew}, it suffices to show that for every purely unrectifiable set $F \subset E$ the quantity $\mathcal{H}^m(\pi_V(F))$ vanishes for almost every $m$-plane $V$. Condition \eqref{e:(2)} directly yields the desired result for $F_2(V)$. As for $F_3(V)$, condition \eqref{e:(3)} also leads to $\mathcal{H}^m(\pi_V(F_3(V))) = 0$, since any set with finite $\mathcal{I}^m_1$-measure intersects a typical affine $(n-m)$-plane in only finitely many points. The real difficulty lies in dealing with $F_1(V)$, which relates to condition \eqref{e:(1)}. 

 Since the class of Hausdorff-negligible sets is stable under Lipschitz maps, condition \eqref{e:(1)} immediately implies the desired vanishing of the projection for $F_1(V)$ when $\mu$ coincides with $\mathcal{H}^m$. However, such stability generally fails for $\mathcal{I}^m_1$-negligible sets, as Mattila shown with a counterexample disproving one implication of Vitushkin’s conjecture \cite{mat3}. More concretely, in our context, condition \eqref{e:(1)} ensures that $\mathcal{H}^m(\pi_{V'}(F_1(V))) = 0$ for almost every $m$-plane $V'$, but provides no information about the projection onto $V$ itself. Specifically, it does not imply the “diagonal” estimate
\[
\mathcal{H}^m\big(\pi_V(F_1(V))\big) = 0, \ \ \text{ for almost every $m$-planes $V$}.
\]
This subtle gap is precisely where the main difficulty in understanding problem (Q.1) arises.

The case $p = \infty$ is special because the quantity $\mathrm{Fav}_\infty(E)$ controls the projections:
\begin{equation}
\label{e:ineqiinfty}
\mathrm{Fav}_\infty(E) \geq \mathcal{H}^m(\pi_V(E)), \ \ \text{ for almost every $m$-planes $V$},
\end{equation}
a condition called in \cite{mic} the \emph{weak projection inequality}. It can be proved that such a condition leads to a particular property of the measure $\mathcal{I}^m_\infty$ (see \cite[Theorem 3.3.14]{fed1}): there exists a negligible set of $m$-planes outside of which \emph{every} $\mathcal{I}^m_\infty$-negligible set $E$ satisfies $\mathcal{H}^m(\pi_V(E)) = 0$. The key point is that this exceptional set of planes is independent of the particular set $E$. Since inequality \eqref{e:ineqiinfty} is specific to the case $p = \infty$, a completely different strategy is required for general $p$. 

\vspace{3mm}

\textbf{Novelties in our approach}. To address problem (Q.1), a natural first step is to take into account the multiplicity of intersections between $E$ and all affine $(n-m)$-planes in $\mathbb{R}^n$. More precisely, for each $V \in \mathrm{Gr}(n,m)$, we define a measure $\mu_V$ on $\mathbb{R}^n$ by
\[
\mu_V(B) := \int_V \#(E \cap B \cap \pi_V^{-1}(y)) \, d\mathcal{H}^m(y), \quad \text{for every Borel set } B \subset \mathbb{R}^n.
\]  
Observe that formula \eqref{e:intgeodirectdef} shows that $\mathcal{I}^m_1$ is obtained by averaging the family of measures $\mu_V$ over the Grassmannian $\mathrm{Gr}(n,m)$.

The first key observation is that the measure $\mathcal{I}^m_1$ admits a natural \emph{lifting} to a measure $\hat{\mathcal{I}}^m$ on the product space $\mathbb{R}^n \times \mathrm{Gr}(n,m)$, which simultaneously encodes both spatial and directional information.

 More precisely, define 
\[
\begin{split}
&\zeta(B):= \int_{\Lambda} \mu_V(B) \, d \gamma_{n,m}, \ \ \text{$B \subset \mathbb{R}^n$ Borel} \\
&\hat{\zeta}(A):= \int_{\Lambda} \mu_V(A_V) \, d \gamma_{n,m}, \ \ \text{$A \subset \mathbb{R}^n \times \text{Gr}(n,m)$ Borel}.
\end{split}
\]
Let $\text{Car}_\zeta$ and $\text{Car}_{\hat{\zeta}}$ denote the Carathéodory constructions of $\mathcal{I}^m_1$ and $\hat{\mathcal{I}}_m$ from the set functions $\zeta$ and $\hat{\zeta}$ over Borel coverings of $\mathbb{R}^n$ and $\mathbb{R}^n \times \mathrm{Gr}(n,m)$, respectively. Then the following diagram commutes:
\begin{displaymath}
  \xymatrix{
    {(\mu_V)_{\text{Gr}(n,m)}} \ar[rr]^{\text{Car}_{\hat{\zeta}}} \ar[dr]^{\text{Car}_{\zeta}}
    && {\mathcal{M}(\mathbb{R}^n \times \text{Gr}(n,m))} \ar[dl]^{\pi_{1\sharp}} \\
    & {\mathcal{M}(\mathbb{R}^n)}
  }
\end{displaymath}
which is, the measure $\mathcal{I}^m_1$ can be seen as the marginal of $\hat{\mathcal{I}}^m$ via the projection $\pi_1 \colon \mathbb{R}^n \times \text{Gr}(n,m) \to \mathbb{R}^n$ onto the first component: 
\[
\pi_{1,\sharp}\hat{\mathcal{I}}_m=\mathcal{I}^m_1.
\]
 The construction of the measures $\mathcal{I}^m_1$ and $\hat{\mathcal{I}}_m$ from  set functions defined as averages of Borel regular measures (unlike in the case of Favard length) is now crucial, as it enables the application of the \emph{Disintegration Theorem} \ref{t:disthm}. Indeed, the measure $\hat{\mathcal{I}}_m$ encodes essential information about $\mathcal{I}^m_1$ via disintegration with respect to the projection $\pi_1$:
\begin{equation}
\hat{\mathcal{I}}_m = \eta_x \otimes \pi_{1 \sharp}\hat{\mathcal{I}}_m, \quad \text{$\eta_x$ probability measure on $\text{Gr}(n,m)$ for $\pi_{1 \sharp}\hat{\mathcal{I}}_m$-a.e. $x \in \mathbb{R}^n$}.
\end{equation}
Assuming that a set $E$ has finite $\mathcal{I}^m_p$-measure for some $p > 1$, it can be shown that for $\pi_{1\sharp}\hat{\mathcal{I}}_m$-almost every $x$, \emph{the measure $\eta_x$ is absolutely continuous with respect to $\gamma_{n,m}$} (see Proposition \ref{p:keyprop}). This regularity allows one to treat $\hat{\mathcal{I}}_m$ essentially as a product measure and to apply a version of Fubini’s theorem: the order of integration over $x \in \mathbb{R}^n$ and $V \in \mathrm{Gr}(n, m)$ may effectively be interchanged. 

This framework allows one to reinterpret condition \eqref{e:(1)}, which is a constraint on the product measure $\mu \otimes \gamma_{n,m}$, directly in terms of the measure $\hat{\mathcal{I}}_m$. This reformulation enables estimates to be transferred between the base and fiber variables in the product space. It is worth emphasizing that the absolute continuity of the conditional measures $\eta_x$ is not only sufficient for the rectifiability conclusion of Theorem~\ref{t:fedpronew}, but also necessary. Indeed, for the purely unrectifiable set $E \subset \mathbb{R}^2$ with $0 < \mathcal{I}^1_1(E) < \infty$ constructed in Mattila’s counterexample \cite{mat2}, the corresponding disintegration must involve conditional measures $\eta_x$ that possess a nontrivial singular part for almost every $x \in E$.

This strategy is supported by a structure theorem (see Theorem \ref{t:structure}) that departs from the classical one mentioned earlier and is designed instead to address the specific analytic features of the problem at hand.

 \subsection{Vitushkin's conjecture} 
   \label{ss:vitushkin}
   
In the 1960s, Vitushkin \cite{vit} conjectured that a compact planar set is non-removable, i.e., has positive analytic capacity, if and only if its orthogonal projections have positive length in a set of directions of positive measure, or equivalently, if it has positive Favard length . Significant progress has been made on this conjecture in the case of $\sigma$-finite sets: Calderón~\cite{cal} established the “if” direction, while David~\cite{dav} solved the “only if” direction for finite sets, and later Tolsa~\cite{tol1} completed the $\sigma$-finite case by proving the semi-additivity of analytic capacity. However, Vitushkin’s conjecture does not hold in full generality, as demonstrated by Mattila~\cite{mat1} and later by Jones and Murai~\cite{jonmur} (see also~\cite{joymor}). The latter authors constructed a compact set with strictly positive analytic capacity, but zero Favard length. This leaves open the fundamental question:

\[
\text{\textbf{Vitushkin (Q.2)}: \emph{Does positive Favard length generally imply positive analytic capacity?}}
\]

\vspace{2mm}
Over the past decade, renewed interest in the conjecture has led to notable advances, which can be classified into multi-scale and single-scale results.

\vspace{3mm}

\emph{Multi-scale results.} Closely related to this approach are results that quantify Besicovitch-Federer’s projection theorem in terms of multi-scale conditions on Favard length \cite{dav2}. A breakthrough result by Orponen \cite{orp1} resolves a key conjecture of David-Semmes in the setting of Ahlfors-regular sets working under the notion of \emph{plenty of big projections} (PBP) conditions (see also \cite{martink,dab3}). D\k{a}browski \cite{dab2} recently addressed an open problem posed in \cite{orp1}, obtaining, broadly speaking, a single-scale version of Orponen’s result. This work relies on a more general condition called uniformly large Favard length (ULFL), which requires Favard length to remain uniformly comparable to the radius at every scale and location:
\[
\text{Fav}(K \cap B_r(x)) \gtrsim_C r \quad \text{for all } x \in K \text{ and } 0 < r < \text{diam}(K).
\]
This condition provides a quantitative notion of rectifiability via big pieces of Lipschitz graphs (see \cite[Subsection 1.1]{dab2}). D\k{a}browski's result not only refines Besicovitch’s projection theorem quantitatively but also sheds new light on Vitushkin’s conjecture under the ULFL assumption, improving on earlier work with Villa \cite{dabvil}, which imposed the stronger PBP condition. In \cite{dabvil}, Orponen’s quantitative rectifiability result was used as a black box to show that a Frostmann measure constructed by the authors satisfies a certain flatness condition (see \cite[Section 3]{dabvil}). In contrast, D\k{a}browski’s quantitative rectifiability result allows the author to directly prove the desired flatness condition under the weaker ULFL assumption (see \cite[Corollary 1.11]{dab2} and \cite[Remark 1.9]{dabvil}). In both works, the strategy relies on a powerful technique (see \cite{davsem}): a corona decomposition of the David–Christ lattice into disjoint trees, each supporting a measure that is approximately Ahlfors regular on its tree. In particular, \cite{dab2} proves that the analytic capacity admits a lower bound in terms of the diameter of the set, thereby providing a positive answer to (Q.2) under the ULFL assumption.

\vspace{3mm}

\emph{A single-scale result.} While multi-scale methods rely on uniform geometric regularity across scales, single-scale approaches relax these requirements by focusing on integrability conditions for projections of measures. Among the results in the second class, of particular relevance is the work by Chang and Tolsa \cite{chatol}, which affirms (Q.2) under stronger integrability conditions on the projections of measures carried by $ K $, while relaxing geometric constraints. Specifically, \cite{chatol} assumes the existence of a measure $\mu$ on $K$ such that, when projected onto a line $\ell$, the resulting measure $\pi_\ell \mu$ is absolutely continuous with respect to the one-dimensional Hausdorff measure $\mathcal{H}^1$, and satisfies the following integrability condition:
\begin{equation}
\label{e:chatol}
0 < \int_{\Lambda} \|(\pi_\ell \mu)\|_2^2 \, d\gamma_{2,1}(\ell) < \infty,
\end{equation}
where $\Lambda \subset \text{Gr}(2,1)$ forms an interval in the space of lines through the origin. In the expression above, $\|\cdot\|_2$ denotes the $L^2$-norm of the absolutely continuous part of the measure $\pi_\ell \mu$ with respect to $\mathcal{H}^1$. The inequality in the right-hand side of \eqref{e:chatol} implicitly imposes a positive Favard length condition on $K$, as can be seen from the estimate 
\[
\text{Fav}(K) \gtrsim \frac{1}{\int \|\pi_\ell \mu\|_2^2 \, d\gamma_{2,1}}.
\]
The authors establish analytic capacity estimates using curvature of measures and conical Riesz potentials in connection with corona decompositions. This method yields quantitative lower bounds on analytic capacity providing, in particular cases, a positive answer to (Q.2). It is worth noting that this result does not imply the aforementioned multi-scale result, nor does the converse hold.

\vspace{3mm}

\subsection*{Two new results toward Vitushkin’s conjecture}
A striking consequence of Theorem~\ref{t:fedpronew} is its deep connection with Vitushkin's conjecture. While the formulation of Theorem~\ref{t:fedpronew} already highlights the geometric principles underlying this connection, the argument we develop toward Vitushkin's conjecture actually relies on a slightly more general version of the theorem, which is established later in the paper as Theorem~\ref{t:central}. 

We now present two qualitative new results addressing (Q.2) within natural class of sets lying beyond the $\sigma$-finite framework: those in the complex plane having finite integral geometric measure. This approach introduces a novel perspective grounded in integral geometry, departing from previously considered techniques. 
\vspace{3mm}

\emph{Multi-scale result}. 
Our first result is a multi-scale criterion that extends previous ULFL-type conditions to the broader class of sets with finite integral geometric measure. To state it precisely, we introduce the notion of $\text{Fav}_\Phi$, defined as in \eqref{e:favp}, where the usual $L^p$-norm is replaced by the Orlicz norm associated with a given Orlicz function $\Phi \colon [0,\infty) \to [0,\infty)$ (see \eqref{e:phinorm}). The $L^p$ case appears as a special instance of this more general framework. This approach allows us to account for all possible super-linear growth behaviors at infinity, leading to a sharper result. We refer the reader to Subsection~\ref{ss:orlicz} for the background on the theory of Orlicz functions relevant to this work, and to Subsection~\ref{ss:msresult} for the proof of the result stated below.

\begin{theorem}
\label{t:multiscaleintro}
    Let $K \subset \mathbb{C}$ be a compact set with $\mathcal{I}^1_1(K) < \infty$. Assume that for every $x \in K$ there exist scales $r_{x,i} \searrow 0$ such that
    \begin{equation}
    \label{e:msfav1intro}
        \emph{Fav}(K \cap B_{r_{x,i}}(x)) >  k_x \, \emph{Fav}_{\Phi}(K \cap B_{r_{x,i}}(x)), \ \ \text{ for every $i=1,2,\dotsc$},
    \end{equation}
    for some constant $k_x$ that might depend on $x \in K$ and for some Orlicz function $\Phi$ with superlinear growth at infinity. Then $K$ has strictly positive analytic capacity.
\end{theorem}

Theorem \ref{t:multiscaleintro} can be derived from Theorem~\ref{t:fedpronew} together with Calderón's work~\cite{cal}. 
Notably, the comparability condition in~\eqref{e:msfav1intro} between the classical Favard length and its generalized counterpart represents a substantial step forward in the understanding of Vitushkin's conjecture, yielding a sharper result for sets with finite integral geometric measure. Indeed, we point out that the comparison required in our theorem applies only to a particular sequence of scales, which may vary from point to point. Beyond this, the key advancement lies in the structural nature of our assumption, which we now turn to examine.

To explain this, we define
\[
\theta_K(\ell,x,r) := \frac{\mathcal{H}^1\big(\pi_\ell(K \cap B_r(x))\big)}{r}, \quad \text{for } (\ell,x,r) \in \mathrm{Gr}(2,1) \times K \times (0,\infty).
\]

The interesting regimes, which cannot be treated under the ULFL condition, are precisely those in which $\theta_K(\ell,x,r) \to 0$ as $r \to 0^+$ for almost every lines $\ell \in \text{Gr}(2,1)$. For instance, consider the following schematic asymptotic behavior of the Favard profile at $x$, in which $\theta_K(\cdot,x,r)$ takes only two values for each $r >0$:
\begin{equation}
\theta_K(\ell,x,r) =
\begin{cases}
    1  & \text{if } \ell \in \Lambda_r \subset \mathrm{Gr}(2,1), \\
    c_r & \text{if } \ell \in \mathrm{Gr}(2,1) \setminus \Lambda_r,
\end{cases}
\end{equation}
for a family of subsets $(\Lambda_r)_{r > 0}$ of $\mathrm{Gr}(2,1)$ and a family of constants $(c_r)_{r > 0}$ satisfying
\begin{equation}
\label{e:constraintmsintro}
 c_r = o(\gamma_{2,1}(\Lambda_r)), \quad \text{as $r \to 0$}.
\end{equation}

In this regime, one can always construct an Orlicz function $\Phi$ with superlinear growth at infinity such that \eqref{e:msfav1} holds (see the end of Subsection~\ref{ss:msresult}). 

We observe that, in the above example, the rate at which $c_r$ converges to zero governs the admissible decay rate of $\gamma_{2,1}(\Lambda_r)$. For instance, if $c_r \equiv 0$ for all sufficiently small $r$, then condition \eqref{e:constraintmsintro} forces the function $\theta_K(\cdot, x, r)$ to be almost everywhere identically zero. This scenario is particularly relevant in relation to the set $K$ constructed in Mattila’s counterexample. Indeed,  by construction, the set $K$ admits a sequence of finite coverings $\mathcal{U}(n)$, each consisting of parallelograms whose diameters shrink to zero as $n \to \infty$, and satisfying
\[
\mathcal{H}^1\big(\pi_\ell(K \cap P)\big) = 0, \quad \text{for every } P \in \mathcal{U}(n),
\]
for every line $\ell \in \mathrm{Gr}(2,1) \setminus \Lambda_n(P)$. As the scale decreases, $K \cap P$ becomes invisible from an increasing range of directions, that is, its projections vanish on larger and larger sets of lines. The exceptional sets $\Lambda_n(P)$ still have small but nonzero $\gamma_{2,1}$-measure, converging to zero as $n \to \infty$ (see \cite[Section 5]{mat2}). Consequently, condition~\eqref{e:constraintmsintro} fails, illustrating that the loss of comparability in~\eqref{e:msfav1intro} is tied to this type of degenerate decay.

We conclude this part by noting that our result draws attention to critical regimes in the Favard profile, specifically, those in which the comparability of Favard lengths breaks down for growth rates exceeding linear behavior at infinity. This suggests that gaining insight into Vitushkin's conjecture may depend on understanding the behavior in this delicate range.

   \vspace{3mm}
   
   \emph{Single-scale result}. The derivation of our single-scale result, in connection with Vitushkin's conjecture, involves establishing the first extension of the Besicovitch–Federer projection theorem to non–$\sigma$-finite sets. Since this result holds in arbitrary dimensions and codimensions, we fix two integers $0 < m < n$ for its formulation. The relevant class of sets we consider consists of those $E \subset \mathbb{R}^n$ that intersect a typical affine $(n-m)$-plane in $\mathbb{R}^n$ in finitely many points. Here, the term ``typical'' refers to any isometrically invariant measure on the space of all affine $(n-m)$-planes in $\mathbb{R}^n$ (see, for instance, \cite[Chapter 3]{mat1}). Our main result reads as follows; the proof can be found in Subsection~\ref{ss:ssresult}.

\begin{theorem}[Extended Besicovitch–Federer projection theorem]
\label{t:besfedintro}
    Let $E \subset \mathbb{R}^n$ be a Borel set intersecting a typical affine $(n-m)$-plane into finitely many points. Then, there exists a countably $m$-rectifiable Borel set $R \subset \mathbb{R}^n$ such that the purely $(\mathcal{H}^m,m)$-unrectifiable set $E \setminus R$ satisfies
    \begin{equation}
    \label{e:spproperty}
        \pi_{V \sharp}\big(\mu \restr (E \setminus R)\big) \, \bot \, \mathcal{H}^m, \ \ \text{ for $\gamma_{n,m}$-a.e. $V \in \emph{Gr}(n,m)$,}
    \end{equation}
    whenever $\mu$ is a finite measure. 
\end{theorem}
Clearly, any set with finite $m$-dimensional integral geometric measure satisfies the finite slicing condition required by the hypothesis of the above theorem. Hence, the result stated above asserts, in particular, that within the class of sets with finite integral geometric measure, purely unrectifiable sets can still be detected via orthogonal projections, provided one allows projecting not the set itself, but every measure supported on it. This is, in a sense, a sharp result, as Mattila's previously discussed counterexample provides a purely unrectifiable set $K \subset \mathbb{R}^2$ such that $0 < \mathcal{I}^1_1(K) < \infty$. 

We limit ourselves here to observing that the proof of this result relies on a general rectifiability criterion for measures admitting an integral geometric representation, which we introduce in Section~\ref{s:rectintgeo}.

Coming back to Vitushkin's conjecture, we see that Theorem~\ref{t:besfedintro} together with Calder\'on's result, immediately implies the following corollary, whose proof is presented in Subsection \ref{ss:ssresult}:

\begin{corollary}
\label{c:singlescale}
    Let $K \subset \mathbb{C}$ be a compact set intersecting a typical line into finitely many points. Then, if there exists a finite measure $\mu$ supported by $E$ such that
    \begin{equation}
    \label{e:ssfav1}
    \pi_{\ell \sharp} \mu^a \neq 0, \ \ \text{ for every $\ell \in \Lambda \subset \emph{Gr}(2,1)$ with $\gamma_{2,1}(\Lambda)>0$},
    \end{equation}
    then $K$ has strictly positive analytic capacity.
\end{corollary}

In the above corollary, the symbol $\pi_{\ell \sharp} \mu^a$ denotes the absolutely continuous part of the projected measure $\pi_{\ell \sharp} \mu$ with respect to $\mathcal{H}^1$. This single-scale result, valid among sets satisfying a finite slicing condition, refines the one presented in \cite{chatol} in two key respects: first, the set of directions $\Lambda$ is only required to have positive measure, without any regularity assumptions; second, the $L^2$-integrability condition on the projection is relaxed to the weaker requirement that its absolutely continuous part be nontrivial.

\subsection{Our central result and plan of the paper}
\label{ss:central-plan}
All the results presented above ultimately derive from a general rectifiability criterion for a class of measures that admit an integral geometric representation (Theorem~\ref{t:central}). While the definition of this class requires a more technical setup, we defer the precise statement to Section~\ref{s:rectintgeo}, where the criterion is developed in full detail. Importantly, the main result applies not only to families of orthogonal projections, but also to more general families of maps that satisfy a transversality condition (see Subsection \ref{ss:transversal}). This broader framework allows us to extend the applicability of our results to a wider class of metric spaces, even in the absence of a natural group action or an associated invariant measure, which are typically present in the framework of homogeneous spaces (see, for instance, \cite{hov,hov1,bal}).

\vspace{3mm}

The paper is organized as follows. In Section~\ref{s:notation}, we recall some notation and review classical results relevant to our analysis. Section~\ref{s:structure} establishes a structure theorem for sets, specifically tailored to the needs of our framework (Theorem~\ref{t:structure}). In Section~\ref{s:rectintgeo}, we rigorously define the class of integral geometric measures and present the proof of our main result (Theorem~\ref{t:central}). Section~\ref{s:application} is devoted to the solution of Problem (Q.1). Section~\ref{s:vitushkin} applies our main result to Vitushkin's conjecture. This section is further divided into two subsections, \ref{ss:ssresult} and \ref{ss:msresult}. The first presents an extension of the Besicovitch-Federer projection theorem to sets satisfying a finite slicing condition, along with its single-scale application to Vitushkin's conjecture. The second subsection contains our multi-scale result on Vitushkin’s conjecture, involving the generalized Favard length. Finally, for the reader’s convenience, we have included three appendices that address technical results not essential to the main flow of the paper but necessary for completeness.

\section{Some preliminary tools }
\label{s:notation}
\subsection{Measures and rectifiability}

Let us briefly fix some notation and recall standard results that will be used throughout the paper.

\begin{definition}
Let $X$ denote a set and $2^X$ the family of all subsets of $X$. A measure on $X$ is a set function $\mu \colon 2^X \to [0,\infty]$ satisfying
\begin{enumerate}
    \item $\mu(\emptyset)=0$
    \item $\mu(E) \leq \sum_{i=1}^\infty \mu(E_i)$ whenever $E \subset \bigcup_{i=1}^\infty E_i$.
\end{enumerate}
\end{definition}

If $X$ is a topological space and every Borel set $B \subset X$ is $\mu$-measurable in the sense of Carathéodory, then $\mu$ is called a Borel measure. We say that $\mu$ is \emph{Borel regular} if for every $E \subset X$ there exists a Borel set $B \supset E$ such that $\mu(B) = \mu(E)$. The measure $\mu$ is \emph{Radon} if it is Borel regular and finite on compact sets.

Given another topological space $Y$ and a Borel map $P \colon X \to Y$, we define the \emph{pushforward measure} $P_\sharp \mu$ on $Y$ by
\begin{equation}
\label{e:pushforward}
    P_\sharp \mu(E) := \inf_{\substack{E \subset B \\ B \subset Y \text{ Borel}}} \mu(P^{-1}(B)), \quad \text{for every } E \subset Y,
\end{equation}
and note that $P_\sharp \mu(B) = \mu(P^{-1}(B))$ for Borel sets $B \subset Y$. Moreover, if $\mu$ is Borel, then so is $P_\sharp \mu$.

If $\mu$ is a $\sigma$-additive set function defined on Borel subsets of $X$ with $\mu(\emptyset)=0$, we consider its extension to a Borel regular measure via:
\begin{equation}
    \mu(E) := \inf_{\substack{E \subset B \\ B \subset X \text{ Borel}}} \mu(B), \quad \text{for every } E \subset X.
\end{equation}

We now recall two standard results for Radon measures. The first concerns density estimates (see \cite[Theorem 6.9]{mat1}). We denote by $\Theta^{*m}(\mu,x)$ the upper $m$-density of $\mu$ at $x$.

A fundamental tool we will employ throughout the paper is the disintegration of measures. We collect its main properties below (see \cite[Theorem 5.3.1]{ags}).

\begin{theorem}[Disintegration Theorem]
\label{t:disthm}
Let $P \colon \mathbb{R}^n \to \mathbb{R}^m$ be a Borel map and let $\mu$ be a finite measure on $\mathbb{R}^n$. Then there exists a $P_\sharp \mu$-a.e.\ uniquely determined family $(\eta_y)_{y \in \mathbb{R}^m}$ of probability measures on $\mathbb{R}^n$ such that
\begin{align}
    \label{e:dis1}
    & \eta_y(\mathbb{R}^n \setminus P^{-1}(y)) = 0, \quad \text{for $P_\sharp \mu$-a.e.\ } y \in \mathbb{R}^m, \\
    \label{e:dis0}
    & y \mapsto \eta_y(B) \text{ is Borel measurable for every Borel set } B \subset \mathbb{R}^n, \\
    \label{e:dis2}
    & \int_{\mathbb{R}^n} f(x)\, d\mu(x) = \int_{\mathbb{R}^m} \left( \int_{\mathbb{R}^n} f(x)\, d\eta_y(x) \right) dP_\sharp\mu(y),
\end{align}
for every Borel function $f \colon \mathbb{R}^n \to [0,\infty]$. Moreover, the measure $\eta_y$ can be explicitly computed as
\begin{equation}
    \label{e:dis1.1.1}
    \frac{1}{P_{\sharp} \mu(B_r(y))}\mu \restr P^{-1}(B_r(y)) \rightharpoonup \eta_y \ \ \text{ weakly in the sense of measures as $r \to 0^+$},
\end{equation}
for $P_{\sharp}\mu$-a.e. $y \in \mathbb{R}^m$.
\end{theorem}

\begin{definition}[Measurable family of measures]
\label{d:measfamily}
We say that a family of measures $(\eta_y)$ is measurable whenever condition \eqref{e:dis0} is satisfied. 
\end{definition}

\begin{definition}[Good representative in the disintegration]
\label{d:goodrepdis}
   Let $P \colon \mathbb{R}^n \to \mathbb{R}^m$ be a Borel map and let $\mu$ be a finite measure on $\mathbb{R}^n$. Then, we define for every $y \in \mathbb{R}^m$ the probability measure $\tilde{\eta}_y$ in $\mathbb{R}^n$ as
  \begin{equation*}
  \tilde{\eta}_y:=
   \begin{cases}
     \lim_{r \to 0^+} \frac{1}{P_{\sharp} \mu(B_r(y))}\mu\restr P^{-1}(B_r(y)), \ &\text{ if $y \in \text{supp}(P_{\lambda\sharp}(\mu))$ and the weak limit exists} \\
     0 \ &\text{ otherwise}.
   \end{cases}
   \end{equation*}
\end{definition}

We now recall the definitions of countably $m$-rectifiable sets and purely $(\mu,m)$-unrectifiable sets, which we will use throughout the paper.

\begin{definition}[Countably $m$-rectifiable]
Let $R \subset \mathbb{R}^n$ be a set. We say that $R$ is countably $m$-rectifiable if and only if
\begin{equation}
\label{e:intro7}
R= \bigcup_i f_i(E_i),
\end{equation}
where $E_i \subset \mathbb{R}^m$ are bounded sets and $f_i \colon E_i \to \mathbb{R}^n$ is Lipschitz for every $i=1,2,\dotsc$.
\end{definition}

\begin{definition}[Purely $(\mu,m)$-unrectifiable]
\label{d:unrect}
Let $\mu$ be a measure in $\mathbb{R}^n$. We say that a set $E \subset \mathbb{R}^n$ is purely $(\mu,m)$-unrectifiable if and only if $\mu(\Sigma \cap E) = 0$ for every countably $m$-rectifiable set $R \subset \mathbb{R}^n$. 
\end{definition}

We conclude this subsection with the definition of $m$-rectifiability for measures.

\begin{definition}[Rectifiable measure]
Let $\mu$ be a Radon measure in $\mathbb{R}^n$. We say that $\mu$ is $m$-rectifiable if and only if $\mu \ll \mathcal{H}^m \restr R$ for some countably $m$-rectifiable Borel set $R \subset \mathbb{R}^n$. 
\end{definition}

\subsection{Transversality and cones}
\label{ss:transversal}

The notion of \emph{transversal family of maps} will play a fundamental role along this section. It was originally introduced in \cite[Definition 7.2]{per} for families of maps $P_\lambda \colon X \to \mathbb{R}^m$ parametrized by $\lambda$ and defined on a compact metric space $(X,\text{d})$. For our purposes it will be sufficient to consider the case when $X$ coincides with $\mathbb{R}^n$. Our definition of transversality can be found in \cite{hov} except for the fact that we do not assume lipschitzianity with respect to the spatial variable $x \in \mathbb{R}^n$. 

\begin{definition}
\label{d:transversal}
Let $n,m,l$ be integers satisfying $0 <m < n$, let $\Lambda \subset \mathbb{R}^l$ be open and bounded, and let $P \colon \mathbb{R}^n \times \Lambda \to \mathbb{R}^m$ be continuous map. Setting $P_{\lambda}(x):=P(x,\lambda)$, we say that the family $(P_\lambda)_{\lambda \in \Lambda}$ is transversal if and only if the following conditions hold true.
\begin{enumerate}[(H.1)]
    \item For every $x \in \mathbb{R}^n$ the map $\lambda \mapsto P_\lambda(x)$ belongs to $C^2(\Lambda;\mathbb{R}^m)$ and
    \begin{equation}
    \label{e:h1}
    \sup_{(\lambda,x) \in \Lambda \times \mathbb{R}^n}\|D^j_\lambda P_\lambda(x)\| < \infty, \ \ \text{for }j=1,2.
    \end{equation}
    \item For $\lambda \in \Lambda$ and $x,x' \in \mathbb{R}^n$ with $x \neq x'$ define
    \begin{equation}
        T_{xx'}(\lambda) := \frac{P_\lambda(x) - P_\lambda(x')}{|x-x'|};
    \end{equation}
    then there exists a constant $C' >0$ such that the property
    \begin{equation}
    \label{e:h2}
        |T_{xx'}(\lambda)| \leq C' \ \ \ \text{ implies } \ \ \
        |\text{J}_\lambda T_{xx'}(\lambda)| \geq C'.
    \end{equation}
   \item There exists a constant $C'' >0$ such that 
   \begin{equation}
   \label{e:h3}
     \|D_\lambda T_{xx'}(\lambda)\|, \|D^2_\lambda T_{xx'}(\lambda)\| \leq C'',
   \end{equation}
   for $\lambda \in \Lambda$ and $x,x' \in \mathbb{R}^n$ with $x \neq x'$.
\end{enumerate}
\end{definition}

Here we need also to define cones around the preimages of points with respect to the maps $(P_\lambda)$. 

\begin{definition}[Cone 1]
\label{d:cone1}
Let $\lambda \in \Lambda$, $a \in \mathbb{R}^n$, $0<s<1$, and $r>0$, we define
\begin{align}
 \label{e:cone1}
    &X(a,\lambda,s) := \{x \in \mathbb{R}^n \ | \ |P_\lambda(x) -P_\lambda(a)| < s\, |x-a|  \},\\
    &X(a,r,\lambda,s) := X(a,\lambda,s) \cap \overline{B}_r(a).
\end{align}
\end{definition}
\begin{definition}[Cone 2]
\label{d:cone2}
Let $\lambda \in \Lambda$, $a \in \mathbb{R}^n$, $0<s<1$, and let $V \subset \mathbb{R}^l$ be an $m$-dimensional plane, we define
\begin{align}
\label{e:cone2}
    &L_V(a,\lambda,s) := \{x \in \mathbb{R}^n \ | \ P_{\lambda'}(x)-P_{\lambda'}(a) =0, \ |\lambda'-\lambda| < s, \ \pi_{V^\bot}(\lambda'-\lambda)=0 \},
\end{align}
where $V^\bot$ is the orthogonal to $V$.
\end{definition}

 In the studying of rectifiability, the key property of transversality relies on the equivalence between the two definitions of cones introduced above. This is the content of the following proposition.
 \begin{proposition}
 \label{p:equivcone}
 Let $(P_\lambda)$ be a family of transversal maps and let $a \in \mathbb{R}^n$, $\lambda_0 \in \Lambda$, and $\delta_0 >0$ such that $\overline{B}_{2\delta_0}(\lambda_0) \subset \Lambda$. If we denote by $B:=\{e_1,\dotsc,e_l\}$ an orthonormal basis of $\mathbb{R}^l$ and by $\mathcal{V}_m$ the family of all $m$-dimensional coordinate planes, there exists $c >0$ and $s_0 >0$ such that, for every $\lambda \in \overline{B}_{\delta_0}(\lambda_0)$, $s < s_0$, and $r>0$, the following inclusions hold true 
 \begin{equation}
 \label{e:equicone}
\bigcup_{V \in \mathcal{V}_m} \overline{B}_r(a) \cap L_{V}(a,\lambda,s/c) \setminus \{a\} \subset X(a,r,\lambda,s) \subset \bigcup_{V \in \mathcal{V}_m} \overline{B}_r(a) \cap L_{V }(a,\lambda,c\,s) \setminus \{a\}.
 \end{equation}
 \end{proposition}
\begin{proof}
The first inclusion follows straightforwardly from the the lipschitzianity of $\lambda \mapsto T_{xx'}(\lambda)$. The second inclusion can be deduced from \cite[Lemma 3.3]{jar}.
\end{proof}

\begin{remark}
\label{r:tratopro}
Since condition \eqref{e:equicone} is local in $\Lambda$, one can work with a family of maps $P_q \colon \mathbb{R}^n \to \mathbb{R}^m$ parametrized by an $\ell$-dimensional manifold $M$, using local charts $(U, \psi)$. To extend the results of this paper to such a setting, it suffices to verify that the family $(P_{\psi^{-1}(\lambda)})_{\lambda \in \psi(U)}$ satisfies Definition \ref{d:transversal}, or, more directly, that the inclusions in \eqref{e:equicone} hold with $\Lambda = \psi(U)$. In this context, the Lebesgue measure $\mathcal{L}^\ell$ should be replaced by (a constant multiple of) the volume measure on $M$.

With this approach, the family of orthogonal projections $(\pi_V)_{V \in \mathrm{Gr}(n,m)}$ can be represented as a finite union of transversal families (see Appendix \ref{a:protra}). In this case, the Lebesgue measure $\mathcal{L}^\ell$ is replaced by the Haar measure $\gamma_{n,m}$ on the Grassmannian $\mathrm{Gr}(n,m)$.  
\end{remark}

\subsection{Orlicz space and $L^1$-weak compactness}
\label{ss:orlicz}

We start by recalling the relevant class of Orlicz functions as well as the associated Orlicz space. For a comprehensive account on this subject we refer to \cite{rya}.

\begin{definition}[Orlicz function]
\label{d:orlicz}
    We say that $\Phi \colon [0,\infty) \to [0,\infty)$ is a Orlicz function, if 
    \begin{enumerate}
        \item $\Phi(0)=0$
        \item $\Phi(t)>0$ if $t>0$ 
        \item $\Phi$ is convex
        \item $\Phi$ is continuous at $0$.
    \end{enumerate}
     Furthermore, we say that $\Phi$ has superlinear growth at infinity if
    \begin{equation}
        \lim_{t \to \infty} \frac{\Phi(t)}{t} = \infty.
    \end{equation}
\end{definition}

\begin{remark}
\label{r:contOrlicz}
    We observe that, since a convex function $\Phi$ is locally Lipschitz on its domain $\text{dom}(\Phi):= \text{int}(\{t \ | \ |\Phi(t)| < \infty \})$, where $\text{int}(\cdot)$ denotes the internal part, we deduce from property (4) that a Orlicz function is continuous on $[0,\infty)$. Moreover, by applying the monotonicity of incremental quotients for convex functions,
\[
\frac{\Phi(t_2) - \Phi(t_1)}{t_2 - t_1} \geq \frac{\Phi(t_1) - \Phi(0)}{t_1}, \quad \text{for all } 0 < t_1 < t_2,
\]
we immediately deduce from (1)-(2) that $\Phi(t_2) \geq \Phi(t_1)$; that is, $\Phi$ is an increasing function.

\end{remark}

Given a Orlicz function $\Phi$ and a finite measure $m$ on some measurable space $X$, we can consider the associated Orlicz vector-space $L_\Phi$ defined as
\begin{equation}
    \bigg\{f \colon X \to \mathbb{R} \ | \ f\text{ is measurable,}\ \int_X \Phi\bigg(\frac{|f|}{c}\bigg) \, dm < \infty \ \text{for some $c>0$} \bigg\}.
\end{equation}
The vector-space $L_\Phi$ can be endowed with the Luxemburg-norm $\|\cdot\|_\Phi$ defined as
\begin{equation}
\label{e:phinorm}
    \|f\|_\Phi := \inf \bigg\{c >0 \ | \ \int_X \Phi\bigg(\frac{|f|}{c}\bigg) \leq 1 \bigg\}, \ \ \text{ for $f \in L_\Phi$}.
\end{equation}

We need the following compactness result.

\begin{proposition}
\label{p:dunford-pettis}
    Let $X$ and $m$ be as above, and let $\Phi$ be a Orlicz function with superlinear growth at infinity. Assume that $(f_r)_{r>0} \subset L_\Phi$ satisfies 
    \[
    \sup_{r>0} \|f_r\|_\Phi < \infty.
    \]
    Then $(f_r)_{r >0}$ is relatively weakly sequentially compact in $L^1(X;m)$.
\end{proposition}

\begin{proof}
 If $\limsup_{r \to 0^+} \|f_r\|_\Phi=0$, then $f_r \to 0$ strongly in $L^1(X;m)$ as $r \to 0^+$. Indeed, thanks to the superlinear growth of $\Phi$, we find $M>0$ such that $\Phi(t) \geq t$ for every $t \geq M$. Therefore, for every $r>0$ we have
\[
\int_{\{f_r \geq (\|f_r\|_\Phi + r)M \}} \frac{|f_r|}{\|f_r\|_\Phi + r} \, dm \leq \int_{\{f_r \geq (\|f_r\|_\Phi + r)M \}} \Phi\bigg(\frac{|f_r|}{\|f_r\|_\Phi + r} \bigg) \, dm \leq 1.
\]
Hence, with the following estimate
\begin{align*}
\int_X |f_r|\, dm &= \int_{\{f_r \geq (\|f_r\|_\Phi + r)M \}} |f_r|\, dm + \int_{\{f_r < (\|f_r\|_\Phi + r)M \}} |f_r|\, dm \\
&\leq \|f_r\|_{\Phi} +r +m(X)(\|f_r\|_\Phi + r)M \\
&= (\|f_r\|_{\Phi} +r)(1+m(X)M),
\end{align*}
we can immediately infer that $f_r \to 0$ in $L^1(X;m)$ as $r \to 0^+$. 

So let us assume that there exists a subsequence such that $\inf_{i} \|f_{r_i}\|_{\Phi} >0$. In this case, it is enough to prove that the sequence $(f_i)_{i}$ defined as $f_i:= f_{r_i}/\|f_{r_i}\|_\Phi$ is relatively weakly sequentially compact in $L^1(X;m)$. Since $\|f_i\|_\Phi=1$ for every $i=1,2,\dotsc$, by recalling that $\Phi$ is increasing (see Remark \ref{r:contOrlicz}), we can write for every $\epsilon$
\[
\int_X \Phi\bigg(\frac{|f_i|}{2} \bigg) \, dm  \leq 1, \text{ for every $i=1,2,\dotsc$}. 
\]
Therefore, by virtue of the superlinear growth assumption on $\Phi$, we can apply Dunford-Pettis's Theorem to deduce that the sequence $(f_i)_i$ is relatively weakly sequentially compact in $L^1(X;m)$. This concludes the proof. 
\end{proof}

\section{A structure theorem}
\label{s:structure}

This section is devoted to the proof of the following structure theorem.

\begin{theorem}[Structure]
\label{t:structure}
Let $(P_\lambda)$ be a family of transversal maps, let $(\sigma_\lambda)$ be measures on $\mathbb{R}^m$ absolutely continuous with respect to $\mathcal{L}^m$, and let $\phi$ be a Borel regular measure on $\mathbb{R}^n$. Given $E \subset \mathbb{R}^n$ $\phi$-measurable with $\phi(E)<\infty$ we have
\begin{equation}
    \label{e:dec}
    E=R \cup (E \setminus R),  
\end{equation}
where $R$ is countably $(\phi,m)$-rectifiable and $\phi$-measurable, and $E \setminus R$ is purely $(\phi,m)$-unrectifiable and $\sigma$-compact. Moreover, if there exists a set $\hat{E} \subset \Lambda \times E$ such that for $\mathcal{L}^l$-a.e. $\lambda \in \Lambda$ we have
\begin{align}
  \label{e:structure1}
  &\sigma_\lambda(P_\lambda(S_\lambda \cap \hat{E}_\lambda))=0\ \text{ whenever } \ S \subset \Lambda \times \mathbb{R}^n \ \text{ and } \
        (\mathcal{L}^l \otimes \phi\restr E) (S)=0 ,\\
    \label{e:structure2}    
   &\mathcal{H}^0(E \cap P^{-1}_\lambda(y))< \infty,  \ \ \text{ for }\sigma_\lambda\text{-a.e. } y \in \mathbb{R}^m,
\end{align}
the following property holds true
\begin{equation}
    \label{e:structure3}
    \sigma_\lambda\big(P_\lambda(\hat{E}_\lambda  \setminus R)\big) =0, \ \ \text{ for $\mathcal{L}^l$-a.e. }\lambda \in \Lambda. 
\end{equation}
\end{theorem}

The proof of the above theorem relies on several intermediate propositions. Since their proofs follow from techniques developed in \cite{fed1} and \cite{hov}, and are not essential for the understanding of the central result, we choose to state them without proof and defer their detailed verification to Appendix \ref{a:structure}. In what follows, we provide only the proof of Theorem~\ref{t:structure}.

Throughout the remainder of this section, we implicitly assume that $(P_\lambda)$ and $\phi$ satisfy the hypotheses of Theorem~\ref{t:structure}.

\begin{proposition}
\label{p:e1delta}
Let $E \subset \mathbb{R}^n$ be a purely $(\phi,m)$-unrectifiable Borel set  with $\phi(E)<\infty$ and let $\delta >0$ and $\lambda \in \Lambda$. Let the set $E_{1,\delta}(\lambda)$ be defined as
\begin{equation}
    \label{e:e1delta}
    E_{1,\delta}(\lambda) := \{a \in \mathbb{R}^n \ | \  \limsup_{s \to 0^+} \, \sup_{0<r<\delta} (rs)^{-m}\phi(E \cap X(a,r,\lambda,s))=0\}.
\end{equation}
Then $\phi(E_{1,\delta}(\lambda) \cap E)=0$.
\end{proposition}

\begin{proposition}
\label{p:e2delta}
Let $E \subset \mathbb{R}^n$ be $\phi$-measurable with $\phi(E) < \infty$ and let $\delta >0$, $\lambda \in \Lambda$. Let the set $E_{2,\delta}(\lambda)$ be defined as
\begin{equation}
    \label{e:e2delta}
    E_{2,\delta}(\lambda) := \{a \in \mathbb{R}^n \ | \ \limsup_{s \to 0^+} \sup_{0<r<\delta} (rs)^{-m} \phi(E \cap X(a,r,\lambda,s))=\infty  \},
\end{equation}
then $\mathcal{L}^m\big(P_\lambda(E_{2,\delta}(\lambda))\big)=0$.
\end{proposition}

The three alternatives contained in the following proposition are crucial for the structure theorem. They were originally proved in \cite[Theorem 3.3.4]{fed1} in a slightly different form.

\begin{proposition}
\label{p:keycond}
Let $E \subset \mathbb{R}^n$ be a $\phi$-measurable set with $\phi(E)<\infty$ and let $\delta >0$. For every $a \in \mathbb{R}^n$, for $\mathcal{L}^l$-a.e. $\lambda \in \Lambda$ one of the following conditions holds true
\begin{align}
    \label{e:cond1}
    &\limsup_{s \to 0^+} \sup_{0<r<\delta}  \phi(E\cap X(a,r,\lambda,s)) (rs)^{-m}=0,\\
    \label{e:cond2}
    &\limsup_{s \to 0^+} \sup_{0<r<\delta}  \phi(E  \cap X(a,r,\lambda,s)) (rs)^{-m}=\infty,\\
    \label{e:cond3}
    &(E \setminus \{a\}) \cap  P^{-1}_\lambda(P_\lambda(a)) \cap \overline{B}_\delta(a) \neq \emptyset. 
\end{align}
\end{proposition}

As an immediate consequence we have the validity of the following lemma.

\begin{lemma}
\label{l:conddelta}
Let $E \subset \mathbb{R}^n$ be a $\phi$-measurable set with $\phi(E)<\infty$ and let $\delta >0$. For $\mathcal{L}^l$-a.e. $\lambda \in \Lambda$, for $\phi \restr E$-a.e. $a \in \mathbb{R}^n$, one of the following conditions holds true
\begin{align}
    \label{e:conddelta1}
    &  \limsup_{s \to 0^+} \sup_{0<r<\delta}  \phi(E  \cap X(a,r,\lambda,s)) (rs)^{-m}=0, \text{ for some }\delta>0,\\
    \label{e:conddelta2}
    &\limsup_{s \to 0^+} \sup_{0<r<\delta}  \phi(E  \cap X(a,r,\lambda,s)) (rs)^{-m}=\infty, \text{ for every }\delta>0,\\
    \label{e:conddelta3}
    & \ \ \ \ \ \ \ \ \ \  (E \setminus \{a\})  \cap P^{-1}_\lambda(P_\lambda(a)) \cap \overline{B}_\delta(a) \neq \emptyset, \text{ for every }\delta>0. 
\end{align}
\end{lemma}
\begin{proof}
First we notice that for every $a \in \mathbb{R}^n$ we have that $\mathcal{L}^l$-a.e. $\lambda \in \Lambda$ satisfies one among \eqref{e:conddelta1}-\eqref{e:conddelta3}. This follows by using the fact that the map 
\[
(0,\infty) \ni \delta \mapsto \limsup_{s \to 0^+} \sup_{0<r<\delta}  \phi(E  \cap X(a,r,\lambda,s)) (rs)^{-m}
\]
is increasing together with Proposition \ref{p:keycond}. Since $\phi \restr E$ is a Radon measure, the same argument at the beginning of the proof of Proposition \ref{p:keycond} tells us that we may equivalently prove the lemma supposing that $E$ is $\sigma$-compact. In this situation we claim that for $(\mathcal{L}^l \otimes \phi \restr E)$-a.e. $(\lambda,a) \in \Lambda \times E$ one between \eqref{e:conddelta1}-\eqref{e:conddelta3} holds true. Clearly our claim follows from Tonelli's theorem \cite[Proposition 5.2.1]{coh} once that we prove the $(\mathcal{L}^l \otimes \phi \restr E)$-measurability of the set of points $(\lambda,a) \in \Lambda \times E$ satisfying \eqref{e:conddelta1} or \eqref{e:conddelta2} or \eqref{e:conddelta3}. This last fact can be obtained as in the last part of the proof of Theorem \ref{t:structure} below.  
\end{proof}

We are now in position to prove Theorem \ref{t:structure}.

\begin{proof}[Proof of Theorem \ref{t:structure}.]
Let us denote by $\mathcal{R}$ the family of all countably $(\phi,m)$-rectifiable and $\phi$-measurable subsets of $\mathbb{R}^n$ and consider the minimum problem
\begin{equation}
\label{e:structure4}
    \min_{R \in \mathcal{R}} \phi(E \setminus R).
\end{equation}
Decomposition \eqref{e:dec} follows immediately if we show that \eqref{e:structure4} admits a solution. In addition we claim that we can find a solution to \eqref{e:structure4} which guarantees also that $E \setminus R$ is $\sigma$-compact. To show this consider a minimizing sequence $(R_j)$ in $\mathcal{R}$ and define $R := \bigcup_j R_j$. Clearly $R \in \mathcal{R}$ and $\phi(E \setminus R) \leq \phi(E \setminus R_j)$ for $j=1,2,\dotsc$, so that $R$ solves \eqref{e:structure4}. Moreover being $E \setminus R$ $\phi$-measurable and being $\phi \restr E$ a Radon measure, we can find a $\sigma$-compact set $\Sigma \subset E \setminus R$ with $\phi((E \setminus R) \setminus \Sigma)=0$. Redefining $R$ with abuse of notation as $R \cup [(E \setminus R) \setminus \Sigma]$ we finally obtain the desired decomposition.

To prove the remaining part we may assume without loss of generality that $E$ is a purely $(\phi,m)$-unrectifiable $\sigma$-compact set. Proposition \ref{p:e1delta} applied to the purely $(\phi,m)$-unrectifiable set $E$ tells us that for every $\lambda \in \Lambda$
\begin{equation}
\label{e:structure8}
\phi(E _{1,\delta}(\lambda) \cap E)=0.
\end{equation}
Define
\begin{equation}
\label{e:structure6}
    S_\delta := \{(\lambda,a) \in \Lambda \times E \ | \ \limsup_{s \to 0^+} \sup_{0 <r<\delta} (rs)^{-m} \phi(E \cap X(a,r,\lambda,s))=0 \}
\end{equation}
and suppose for a moment that $S_\delta$ is $(\mathcal{L}^l \otimes \phi \restr E)$-measurable
 for every $\delta >0$. Since $(S_\delta)_\lambda  = E_{1,\delta}(\lambda) \cap E$ for every $\lambda \in \Lambda$, property \eqref{e:structure8} together with Tonelli's theorem \cite[Proposition 5.2.1]{coh} allow us to deduce 
 \begin{equation}
     (\mathcal{L}^l \otimes \phi \restr E)(S_\delta) =0, \ \ \ \ \delta >0
 \end{equation}
 and hence by setting $S:= \cup_{j=1}^\infty S_{1/j}$ also
 \begin{equation}
     ( \mathcal{L}^l \otimes \phi \restr E)(S) =0.
 \end{equation}
 Condition \eqref{e:structure1} implies thus that for $\mathcal{L}^l$-a.e. $\lambda \in \Lambda$
 \begin{equation}
     \label{e:structure9}
     \sigma_\lambda(P_\lambda(S_\lambda \cap \hat{E}_\lambda))=0. 
 \end{equation}
 Combining this with condition \eqref{e:structure2} we infer that for $\mathcal{L}^l$-a.e. $\lambda \in \Lambda$
 \begin{equation}
     \label{e:structure10}
     \sigma_\lambda(P_\lambda(S_\lambda \cap \hat{E}_\lambda))=0 \ \ \text{and} \ \
     E \cap P_\lambda^{-1}(y) \  \text{is finite for $\sigma_\lambda$-a.e. }y \in \mathbb{R}^m. 
 \end{equation}
 Define
 \begin{equation}
     \label{e:structure11}
     S':=\{(\lambda,a) \in \Lambda \times E  \ | \ \text{\eqref{e:conddelta1}-\eqref{e:conddelta3}} \text{ do not hold true}\},
 \end{equation}
 and suppose for a moment that $S'$ is $(\mathcal{L}^l \otimes \phi \restr E)$-measurable. Lemma \ref{l:conddelta} in combination with Tonelli's theorem tell us that $(\mathcal{L}^l \otimes \phi \restr E)(S')=0$. As a consequence, applying again property \eqref{e:structure1} we deduce that for $\mathcal{L}^l$-a.e. $\lambda \in \Lambda$
 \begin{equation}
     \label{e:structure12}
     \sigma_\lambda(P_\lambda(S'_\lambda \cap \hat{E}_\lambda))=0.
 \end{equation}
 In addition notice that for every $\lambda \in \Lambda$
 \begin{equation}
 \label{e:structure13}
 \begin{split}
    \hat{E}_\lambda  &= \{a \in \hat{E}_\lambda  \ | \ \text{\eqref{e:conddelta1} or \eqref{e:conddelta2} or \eqref{e:conddelta3} holds true}   \} \cup (S'_\lambda \cap \hat{E}_\lambda) \\
    &= \{a \in \hat{E}_\lambda   \ | \ \text{\eqref{e:conddelta2} or \eqref{e:conddelta3} holds true}   \} \cup (S_\lambda \cap \hat{E}_\lambda) \cup (S'_\lambda \cap \hat{E}_\lambda).
    \end{split}
 \end{equation}
 Putting together \eqref{e:structure10}-\eqref{e:structure11}, Proposition \ref{p:e2delta}, and the condition $\sigma_\lambda \ll \mathcal{L}^m$ (a.e. $\lambda$), we infer that for $\mathcal{L}^l$-a.e. $\lambda \in \Lambda$
 \begin{align}
     \label{e:structure14}
     \sigma_\lambda\big(P_\lambda(\{a \in  \hat{E}_\lambda   \ | \ &\text{\eqref{e:conddelta2} or \eqref{e:conddelta3} holds true}   \})\big) =0, \\
     \label{e:structure15}
     &\sigma_\lambda(P_\lambda(S_\lambda \cap \hat{E}_\lambda))=0,\\
     \label{e:structure16}
     &\sigma_\lambda(P_\lambda(S'_\lambda \cap \hat{E}_\lambda))=0.
 \end{align}
 This together with \eqref{e:structure13} and the sub-additivity of $\sigma_\lambda$ gives the desired conclusion.
 
 It remains to prove the measurability of $S_\delta$ and $S'$. To this purpose define the $(\phi \restr E \otimes \mathcal{L}^l \otimes \phi \restr E)$-measurable set $Z$ by
 \[
 Z:= \{(a,\lambda,x) \in E \times \Lambda \times E \ | \ |P_\lambda(x) -P_\lambda(a)| < s|x-a|, \ |x-a| \leq r \},
 \]
and notice that $E \cap X(a,r,\lambda,s) = Z_{(a,\lambda)}$. Hence by Tonelli's theorem the map $(a,\lambda) \mapsto \phi(E \cap X(a,r,\lambda,s))$ is $(\phi \restr E \otimes \mathcal{L}^n)$-measurable for every $r>0$ and $0<s<1$. Since for every $a \in E$ and for every $\lambda \in \Lambda$ we can make use of Lebesgue's dominated convergence theorem to infer 
\[
\lim_{k \to \infty} \phi(E \cap X(a,r_k,\lambda,s_k))= \phi(E \cap X(a,r,\lambda,s)), \ \ r>0, \ 0<s<1,
\]
whenever $r_k \searrow r$ and $s_k \nearrow s$, we can write for $(\mathcal{L}^l \otimes \phi \restr E)$-a.e. $(\lambda,a)$, for every $j=1,2,\dotsc$, and every $\delta>0$
\[
\begin{split}
    \chi_j(a,\lambda)&:=\sup_{0<s<1/j} \sup_{0< r < \delta} (rs)^{-m} \phi(E  \cap X(a,r,\lambda,s))\\ &= \sup_{s \in (0,1/j)\cap \mathbb{Q}} \ \sup_{r \in (0,\delta) \cap \mathbb{Q}} (rs)^{-m} \phi(E \cap X(a,r,\lambda,s)).
\end{split}
\]
But this means that $\chi_j(a,\lambda)$ is $(\mathcal{L}^l \otimes \phi \restr E)$-measurable. Since
\[
S_\delta = \{(a,\lambda) \in E \times \Lambda \ | \ \lim_{j \to \infty} \chi_j(a,\lambda)= 0  \},
\]
we deduce also the $(\mathcal{L}^l \otimes \phi \restr E )$-measurability of $S_\delta$ for every $\delta >0$. 

Now we pass to the measurability of $S'$. Notice that 
\[
S_1:=\{(a,\lambda) \in E \times \Lambda \ | \ \eqref{e:conddelta1} \text{ holds true} \} = \bigcup_{j =1}^\infty S_{1/j}
\]
and this gives the $(\mathcal{L}^l \otimes \phi \restr E)$-measurability of $S_1$. Analogously one can prove the $( \mathcal{L}^l \otimes \phi \restr E)$-measurability of $S_2:= \{(a,\lambda) \in E \times \Lambda \ | \ \eqref{e:conddelta2} \text{ holds true} \}$. Define $S_3 := \{(a,\lambda) \in E \times \Lambda \ | \ \eqref{e:conddelta3} \text{ holds true} \}$ and notice that 
\begin{equation}
\label{e:structure17}
\begin{split}
\pi_{1,2}(\{(a,\lambda,x)& \in E \times  \Lambda \times E \ | \ P_\lambda(x) =P_\lambda(a), \ 0< |x-a| \leq \delta  \})\\
&= \{(a,\lambda) \in E \times \Lambda \ | \ (E \setminus \{a\}) \cap P^{-1}_\lambda(P_\lambda(a)) \cap \overline{B}_\delta(a) \neq \emptyset \}.
\end{split}
\end{equation}
Since $E$ is Borel by assumption, we can apply the measurable projection theorem \cite[Proposition 8.4.4]{coh} to infer that the set in \eqref{e:structure17} is $( \mathcal{L}^l \otimes \phi \restr E)$-measurable and since 
\[
\begin{split}
    S_3= \bigcap_{j=1}^\infty \{(a,\lambda) \in E \times \Lambda \ | \ (E \setminus \{a\})  \cap P^{-1}_\lambda(P_\lambda(a)) \cap \overline{B}_{1/j}(a) \neq \emptyset \},
\end{split}
\]
we deduce that also $S_3$ is $( \mathcal{L}^l \otimes \phi \restr E)$-measurable. Finally, the $( \mathcal{L}^l \otimes \phi \restr E)$-measurability of $S_1,S_2,S_3$ immediately implies the $( \mathcal{L}^l \otimes \phi \restr E)$-measurability of $S'$.
\end{proof}

\section{Rectifiability in the class of integralgeometric measures}
\label{s:rectintgeo}
For the rest of the section we assume that $\Lambda \subset \mathbb{R}^l$ is an open and bounded set and that the family of maps $(P_\lambda)_{\lambda \in \Lambda}$, where $P_\lambda \colon \mathbb{R}^n \to \mathbb{R}^m$, is transversal in the sense of Definition \ref{d:transversal}. 

Given a measurable family of Borel regular measures $(\mu_\lambda)_{\lambda \in \Lambda}$ in $\mathbb{R}^n$ (see Definition \ref{d:measfamily}), and given a Orlicz function $\Phi \colon [0,\infty) \to [0,\infty)$ (see Definition \ref{d:orlicz}), we set
\[
g_B(\lambda):= \int_\Lambda \mu_\lambda(B) \, d\lambda, \ \ \text{ for every Borel set $B \subset \mathbb{R}^n$}.
\]
The associated (outer) Borel regular measure $\mathscr{I}^m_\Phi$ is defined via Caratheodory's construction for every set $E \subset \mathbb{R}^n$ as
\begin{equation}
\label{e:caratheodoryc2}
\mathscr{I}^m_\Phi(E) := \sup_{\delta > 0} \, \inf_{\mathcal{G}_\delta} \sum_{B \in \mathcal{G}_\delta}  \|g_B\|_\Phi,
\end{equation}
where the infimum is taken with respect to all countable coverings $\mathcal{G}_\delta$ of $E$ made of Borel sets with diameter less than or equal to $\delta$.
We are now in position to present the relevant class of integralgeometric measures.   

\begin{definition}
\label{d:igm}
 The measure $\mathscr{I}^m_\Phi$ defined in \eqref{e:caratheodoryc2} is integralgeometric if and only if
\begin{equation}
    \label{e:condigm1}
       P_{\lambda\sharp} \, \mu_\lambda \ll \mathcal{L}^m   \  \text{for $\mathcal{L}^l$-a.e. } \lambda \in \Lambda 
\end{equation}
and there exists a Borel set $E \subset \mathbb{R}^n$ satisfying the following two conditions
\begin{align}
\label{e:condigm5}
& \ \ \ \ \ \ \ \ \ \ \ \ \ \ \ \ \ \ \  \mu_\lambda(\mathbb{R}^n \setminus E)=0 \ \ \text{for $\mathcal{L}^l$-a.e. $\lambda \in \Lambda$}\\
\label{e:condigm6}
&\mathcal{H}^0(E \cap P^{-1}_\lambda(y)) < \infty \ \ \text{for $\mathcal{L}^l$-a.e. }\lambda \in \Lambda \text{ and $P_{\lambda\sharp}\mu_\lambda$-a.e. }y \in \mathbb{R}^m.
\end{align}
\end{definition} 

\begin{remark}
    In order to ease the notation, when $\Phi(t)=t^p$ for some exponent $p\geq 1$, we compactly denote $\mathscr{I}^m_\Phi$ as $\mathscr{I}^m_p$.
\end{remark}

Here we present a few elementary properties of the measures $\mathscr{I}^m_\Phi$ which follows by the construction in \eqref{e:caratheodoryc2}.

\begin{proposition}
\label{p:borelreg}
Let $(\mu_\lambda)_{\lambda \in \Lambda}$ be a measurable family of Borel regular measures in $\mathbb{R}^n$ and let $\Phi$ be a Orlicz function. Then 
\begin{enumerate}
    \item The measure $\mathscr{I}_\Phi^m$ is Borel regular.
    \item  The following equalities hold true
\begin{align}
    \label{e:equal3}
    \mathscr{I}_1^m(B) &= \zeta_1(B)=  \int_{\Lambda} \mu_\lambda(B ) \, d\lambda, \ \ \ \text{\emph{for every} } B \subset \mathbb{R}^n \ \text{\emph{Borel}}.
\end{align}
 \item For every Borel set $B \subset \mathbb{R}^n$ it holds true
 \begin{equation}
 \label{e:zeroiff}
 \text{$\mathscr{I}^m_\Phi(B)=0$ \ \emph{if and only if} \ $\mu_\lambda(B)=0$ for $\mathcal{L}^l$-a.e. $\lambda \in \Lambda$}.
 \end{equation}
 In particular, for every couple of Orlicz functions $\Phi$ and $\Phi'$, and for every set $E \subset \mathbb{R}^n$, it follows
 \begin{equation}
 \label{e:samenegset}
     \mathscr{I}^m_\Phi(E)=0 \ \emph{if and only if}  \  \mathscr{I}^m_{\Phi'}(E)=0.
 \end{equation}
\end{enumerate}

\end{proposition}
\begin{proof}
Condition (1) follows from the use of Borel coverings in the Carathéodory construction (see \cite[Subsection 2.10.1]{fed1}).

Let us prove (2). 
 Given $B \subset \mathbb{R}^n$ Borel and $\epsilon >0$, from the definition of $\mathscr{I}^m_1$ we find a countable Borel covering $(B_i)$ of $B$ satisfying $\sum_i\zeta_1(B_i) \leq \mathscr{I}^m_1(B) + \epsilon$. Therefore
\[
\int_{\Lambda} \mu_\lambda(B) \, d\lambda = \zeta_1(B) \leq \sum_i\zeta_1(B_i) \leq \mathscr{I}^m_1(B) +\epsilon,
\]
where we have used the $\sigma$-subadditivity of $\zeta_1(\cdot)$. The arbitrariness of $\epsilon$ gives
\begin{equation}
\label{e:equal4.1r}
\zeta_1(B) \leq \mathscr{I}^m_1(B), \ \ B \subset \mathbb{R}^n \ \text{Borel}.
\end{equation}
It remains to prove
\begin{equation}
\label{e:equal4.1}
 \zeta_1(B) \geq \mathscr{I}^m_1(B), \ \ B \subset \mathbb{R}^n \ \text{Borel}.
\end{equation}
Without loss of generality we may assume $\zeta_1(B)<\infty$. We notice that, in order to prove \eqref{e:equal4.1}, it is enough to show that for every $\delta>0$ we can find a Borel countable covering of $B$, say $\mathcal{G}_\delta$, made of sets having diameter less than or equal to $\delta$, such that
\begin{equation}
\label{e:equal5.1}
\zeta_1(B) = \sum_{\tilde{B} \in G_\delta} \zeta_1(\tilde{B}).
\end{equation}
But this will easily follow if we prove that $\zeta_1$ is $\sigma$-additive on Borel sets. To show this, we observe that the measure $\zeta_1$ is $\sigma$-additive on Borel sets because it is obtained as the average of the Borel regular measures $(\mu_\lambda)$: indeed, for every pairwise disjoint sequence $(B_k)$ of Borel sets 
\[
\zeta_1(\bigcup_{k}B_k) =\int_\Lambda \mu_\lambda(\bigcup_{k}B_k) \, d\lambda  = \sum_k \int_\Lambda \mu_\lambda(B_k) \, d\lambda = \sum_k \zeta_1(B_k), 
\]
where we have used Beppo Levi's monotone convergence theorem. This concludes the proof of \eqref{e:equal3}.

 Now let us prove (3). Suppose that $\mu_\lambda(B) = 0$ for $\mathcal{L}^l$-a.e. $\lambda \in \Lambda$. Fix $\delta > 0$, and observe that if $(B_i)$ is a Borel covering of $B$ with $\text{diam}(B_i) \leq \delta$ for all $i \in \mathbb{N}$, then the family $(B \cap B_i)$ also forms a Borel covering of $E$ with $\text{diam}(B \cap B_i) \leq \delta$ for all $i$. Set
\begin{equation}
\label{e:notation3.2}
g_{E}(\lambda):= \mu_\lambda(E), \ \ \text{ for $E \subset \mathbb{R}^n$ Borel and $\lambda \in \Lambda$}.
\end{equation}
 Since by assumption $g_B(\lambda) = 0$ for $\mathcal{L}^l$-a.e. $\lambda$, we infer the validity of the following inequality 
 \[
 \sum_i \zeta_\Phi(B \cap B_i) = \sum_i \|g_{B \cap B_i}\|_\Phi \leq \sum_i \|g_{B}\|_\Phi =0,
 \]
 Thanks to the arbitrariness of $\delta >0$, we immediately conclude that $\mathscr{I}^m_\Phi(B)=0$. 

  Suppose now that $\mathscr{I}^m_\Phi(B)=0$. By construction we find Borel coverings of $B$, say $(B_i^k)_{i=1}^{\infty}$ for $k=1,2,\dotsc$, such that $\sum_i \zeta_\Phi(B_i^k) \to 0$ as $k \to \infty$. 
 
  Using the notation in \eqref{e:notation3.2}, we can use the $\sigma$-subadditivity of the $\Phi$-norm (by construction we have $\sum_i \|g_{B^k_i}\|_\Phi < \infty$) to write for every $k=1,2,\dotsc$
 \begin{align*}
 \|g_B\|_\Phi \leq  \bigg\|  \sum_{i=1}^\infty g_{B \cap B^k_i }\bigg\|_\Phi  &\leq    \sum_{i=1}^\infty \|g_{B \cap B^k_i }\|_\Phi\\
   &= \sum_{i=1}^{\infty} \zeta_\Phi(B \cap B_i^k) \\
 &\leq \sum_{i=1}^{\infty} \zeta_\Phi(B_i^k).
 \end{align*}
 Letting $k \to \infty$ in the last term of the previous chain of inequalities, we infer $\|g_B\|_\Phi=0$ which immediately implies the desired result.

 Eventually, condition \eqref{e:samenegset} is a direct consequence of the Borel regularity together with property \eqref{e:zeroiff}. 

\end{proof}

\begin{proposition}
Let $(\mu_\lambda)_{\lambda \in \Lambda}$ be a measurable family of Borel regular measures in $\mathbb{R}^n$ and let $\Phi$ be a Orlicz function. By denoting $g_B(\lambda):= \mu_\lambda(B)$ for every Borel set $B \subset \mathbb{R}^n$ and every $\lambda \in \Lambda$, the measure $\mathscr{I}_\Phi^m$ satisfies
\begin{equation}
\label{e:sigmasub}
      \|g_B\|_\Phi \leq \mathscr{I}_\Phi^m(B), \ \ B \subset \mathbb{R}^n \text{ Borel}.
      \end{equation}
\end{proposition}
\begin{proof}
 It is enough to use the $\sigma$-subadditivity of $\mu_\lambda$ with the subadditivity of the $\Phi$-norm.
\end{proof}

\begin{proposition}
\label{r:inequality}
Let $\Phi$ be a Orlicz function with superlinear growth at infinity. Then
\[
\mathscr{I}_1^m(B) \leq c(\Lambda,\Phi) \,\mathscr{I}_\Phi^m(B), \ \ \text{ for every Borel set $B \subset {\mathbb{R}^n}$},
\]
for some positive constant $c(\Lambda,\Phi)$ depending only on $\mathcal{L}^l(\Lambda)$ and on $\Phi$. 
\end{proposition}
\begin{proof}
     Since $\Phi$ has superlinear growth at infinity, there exists $M > 0$ such that $\Phi(t) \geq t$ for every $t \geq M$. Define $g_B(\lambda) := \mu_\lambda(B)$ for every Borel set $B \subset \mathbb{R}^n$ and every $\lambda \in \Lambda$. 

If $\|g_B\|_\Phi = 0$, then $g_B(\lambda) = 0$ for $\mathcal{L}^l$-a.e.\ $\lambda \in \Lambda$, and hence also $\|g_B\|_1 = 0$. Otherwise, we can write

    \begin{align*}
       \int_\Lambda \frac{g_B(\lambda)}{\|g_B\|_\Phi} \, d\lambda &= \int_{\{g_B < 2\|g_B\|_\Phi M\}} g_B(\lambda) \, d\lambda + \int_{\{g_B \geq 2\|g_B\|_\Phi M\}} g_B(\lambda) \, d\lambda \\
       &\leq  \mathcal{L}^l(\Lambda) 2 M + \int_{\{g_B \geq 2\|g_B\|_\Phi M\}} \Phi\bigg(\frac{g_B(\lambda)}{2\|g_B\|_\Phi}\bigg) \, d\lambda \\
       &\leq  \mathcal{L}^l(\Lambda) 2 M + 1.
    \end{align*}
    Therefore, for every Borel set $B \subset \mathbb{R}^n$ we deduce
    \begin{equation}
        \zeta_1(B)=\int_\Lambda \mu_\lambda(B) \, d\lambda =  \int_\Lambda g_B(\lambda)\, d\lambda \leq \big(\mathcal{L}^l(\Lambda) 2 M + 1 \big) \|g_B\|_\Phi= \big(\mathcal{L}^l(\Lambda) 2 M + 1 \big)  \zeta_\Phi(B).
    \end{equation}
    By setting $c(\Lambda, \Phi) := \mathcal{L}^l(\Lambda) \cdot 2M + 1$, the proposition follows immediately from the construction of the measures $\mathscr{I}^m_\Phi$.

\end{proof}

\subsection{Lifting $\mathscr{I}^m_1$ to a measure on the product space $\mathbb{R}^n \times \Lambda$ }

We are interested in finding a suitable extension of the measure $\mathscr{I}^m_1$, given by \eqref{e:caratheodoryc2}, to the product space $\mathbb{R}^n \times \Lambda$. This can be achieved by following a Fubini's type construction as follows. For every couple of Borel sets $B \subset \mathbb{R}^n$ and $U \subset \Lambda$ we define the set function
\begin{equation}
\label{e:hatzeta.1}
    \hat{\zeta}(B \times U) := \int_U \mu_\lambda(B) \, d\lambda.
\end{equation}
The outer measure $\hat{\mathscr{I}}_m \colon 2^{\mathbb{R}^n \times \Lambda} \to [0,\infty]$ is then defined as
\begin{equation}
\label{e:fubinicon}
  \hat{\mathscr{I}}_m(A):=  \inf \sum_{i=1}^\infty \hat{\zeta}(B_i \times U_i), 
\end{equation}
where the infimum is taken with respect to all sequence of Borel sets $B_i \subset \mathbb{R}^n$ and $U_i \subset \Lambda$ such that
\[
A \subset \bigcup_{i=1}^\infty B_i \times U_i.
\]

In the next proposition, we collect some properties of the measure $\hat{\mathscr{I}}_m$. Although $\hat{\mathscr{I}}_m$ is not a product measure, since one of its factors is a family of measures $(\mu_\lambda)$ rather than a fixed measure, the proof follows the argument used in the proof of Fubini’s theorem in \cite[Subsection 2.6.2]{fed1}. For completeness, we present a detailed proof in Appendix \ref{a:fubini}.

\begin{proposition}[Measurability properties of $\hat{\mathscr{I}}_m$]
\label{p:fubini}
   Let $(\mu_\lambda)_{\lambda \in \Lambda}$ be a measurable family of Borel regular measures in $\mathbb{R}^n$ and let $\Phi$ be a Orlicz function. The measure $\hat{\mathscr{I}}_m$ defined in \eqref{e:fubinicon} satisfies the following properties
    \begin{enumerate}
        \item $\hat{\mathscr{I}}_m$ is a Borel regular measure
        \item If $B \subset \mathbb{R}^n$ and $U \subset \Lambda$ are Borel sets, then
        \begin{equation}
        \label{e:fubprodset}
            \hat{\mathscr{I}}_m(B \times U) = \int_U \mu_\lambda(B) \, d\lambda
        \end{equation}
        \item If $A \subset \mathbb{R}^n \times \Lambda$ is a Borel set with finite $\hat{\mathscr{I}}_m$-measure, then
        \begin{align}
        \label{e:fubprodset1}
           \lambda \mapsto \mu_\lambda(A_\lambda) \text{ is $\mathcal{L}^l$-measurable} \quad \text{ and } \quad
           \hat{\mathscr{I}}_m(A) = \int_{\Lambda} \mu_\lambda(A_\lambda) \, d\lambda.
        \end{align}
    \end{enumerate}
\end{proposition}

\begin{remark}
Formula \eqref{e:fubprodset} immediately implies
\begin{equation}
    \label{e:equal200}
    \pi_{1\sharp}\hat{\mathscr{I}}_m(E) = \mathscr{I}^m_1(E), \ \ E \subset \mathbb{R}^n, 
\end{equation}
where $\pi_1 \colon \mathbb{R}^n \times \Lambda \to \mathbb{R}^n$ dentes the orthogonal projection into the first factor.
\end{remark}

 \begin{remark}
 \label{p:nullset}
 Let $A \subset \mathbb{R}^n \times \Lambda$ be a set. Then, the condition $\hat{\mathscr{I}}_m(A)=0 $ implies $\mu_\lambda(A_\lambda)=0$ for $\mathcal{L}^l$-a.e. $\lambda \in \Lambda$. Indeed, it is enough to use the Borel regularity of the measure $\hat{\mathscr{I}}_m$ together with formula \eqref{e:fubprodset1}.
 \end{remark}
 
The above proposition shows that the measure $\hat{\mathscr{I}}_m$ inherits all the measurability properties of a genuine product measure, except for the key symmetry required by Fubini’s theorem, namely, the possibility of exchanging the order of integration when evaluating the measure of measurable sets. This lack of symmetry is evident in the definition of $\hat{\zeta}$, since the relation defining in \eqref{e:hatzeta.1} is not symmetric in the variables $x$ and $\lambda$, unlike in the classical product measure setting.

However, as we will see in the next result, when the Orlicz function $\Phi$ grows superlinearly at infinity, a Fubini-type property still holds: the order of integration can, in a precise sense, be interchanged. This property is established in the next proposition and plays a fundamental role in the subsequent analysis.

\begin{proposition}
\label{p:keyprop}
Let $(\mu_\lambda)_{\lambda \in \Lambda}$ be a measurable family of Borel regular measures in $\mathbb{R}^n$ and let $\Phi$ be a Orlicz function with superlinear growth at infinity . Assume that the measure $\mathscr{I}_\Phi^m$ constructed as in \eqref{e:caratheodoryc2} is finite. Then, the measure $\hat{\mathscr{I}}_m$ is finite and the disintegration 
\begin{equation}
    \label{e:keyprop1}
    \hat{\mathscr{I}}_m  = \eta_x \otimes \mathscr{I}^m_1,
\end{equation}
satisfies $\eta_x \ll \mathcal{L}^l$ for $\mathscr{I}^m_1$-a.e. $x \in \mathbb{R}^n$.
\end{proposition}
\begin{proof}
 
 The finiteness of $\hat{\mathscr{I}}_m$ follows from the fact that, Proposition \ref{r:inequality} implies $\mathscr{I}^m_1(\mathbb{R}^n) < \infty$, and, as a consequence, formulas \eqref{e:equal3} and \eqref{e:fubprodset} imply $\hat{\mathscr{I}}_m(\mathbb{R}^n \times \Lambda) < \infty$ and $\pi_{1\sharp} \hat{\mathscr{I}}_m = \mathscr{I}^m_1$, where $\pi_1 \colon \mathbb{R}^n \times \Lambda \to \mathbb{R}^n$ denotes the projection onto the first factor. Therefore, the disintegration in \eqref{e:keyprop1} is justified.

 If $x_0 \in \mathbb{R}^n$ we can consider for every $r>0$ such that $B_r(x_0) \subset \mathbb{R}^n$ the $\mathcal{L}^l$-measurable maps $f_r(\lambda):= \mu_\lambda(B_r(x_0))/\mathscr{I}^m_1(B_r(x_0) )$. Notice that by \eqref{e:sigmasub}
\begin{align}
\label{e:keyprop5}
\|f_r\|_{\Phi} = \frac{\zeta_\Phi(B_r(x_0))}{\mathscr{I}^m_1(B_r(x_0))} \leq \frac{\mathscr{I}_\Phi^m(B_r(x_0))}{\mathscr{I}^m_1(B_r(x_0) )}.
\end{align}
Applying Radon-Nikodym's theorem for the two Radon measures $\mathscr{I}_\Phi^m ,\mathscr{I}^m_1 $ we obtain 
\begin{equation}
\label{e:keyprop5.1}
    \lim_{r \to 0^+} \frac{\mathscr{I}_\Phi^m(B_r(x_0) )}{\mathscr{I}^m_1(B_r(x_0))} \text{ exists and is finite for $\mathscr{I}^m_1$-a.e. $x_0 \in \mathbb{R}^n$}.
\end{equation}
Now let $\varphi \in C^{0}_c(\Lambda)$ be non negative and notice that by \eqref{e:fubprodset1}
\begin{equation}
\label{e:keyprop2}
\frac{1}{\mathscr{I}^m_1(B_r(x_0))}\int_{\mathbb{R}^n \times \Lambda} \mathbbm{1}_{B_r(x_0)}\varphi \, d\hat{\mathscr{I}}_m = \int_{\Lambda}\varphi \, \frac{\mu_\lambda(B_r(x_0) )}{\mathscr{I}^m_1(B_r(x_0))}  \, d\lambda =
\int_{\Lambda}\varphi\,  f_r \, d\lambda.
\end{equation}
On the other hand, by using \eqref{e:keyprop1} we know also
\begin{equation}
\label{e:keyprop3}
\frac{1}{\mathscr{I}^m_1(B_r(x_0) )}\int_{\mathbb{R}^n \times \Lambda} \mathbbm{1}_{B_r(x_0)}\varphi \, d\hat{\mathscr{I}}_m = \mint_{B_r(x_0)} \bigg( \int_{\Lambda} \varphi \, d\eta_x \bigg)d\mathscr{I}^m_1 .
\end{equation}
Now consider $D$ a countable dense subset of all non-negative functions in $C^0_c(\Lambda)$. By using \eqref{e:keyprop3} together with Lebesgue's differentiation theorem, for $\mathscr{I}^m_1$-a.e. $x_0 \in \mathbb{R}^n$ and for every $\varphi \in D$ we have  
\begin{equation}
\label{e:keyprop4}
    \lim_{r \to 0^+} \frac{1}{\mathscr{I}^m_1(B_r(x_0))}\int_{\mathbb{R}^n \times \Lambda} \mathbbm{1}_{B_r(x_0)}\varphi \, d\hat{\mathscr{I}}_m = \int_\Lambda \varphi \, d\eta_{x_0}.
\end{equation}
If we look to those $x_0 \in \mathbb{R}^n$ for which \eqref{e:keyprop4} holds true, we can make use of \eqref{e:keyprop5}-\eqref{e:keyprop2} to deduce that up to pass through a not-relabeled subsequence on $r$ (depending on $x_0$), Proposition \ref{p:dunford-pettis} tells us that $f_r \rightharpoonup f$ weakly in $L^1(\Lambda)$, and hence
\begin{equation}
    \label{e:keyprop6}
    \int_{\Lambda} \varphi \, d\eta_{x_0} = 
    \int_{\Lambda} \varphi \, f \, d\lambda, \ \ \ \varphi \in D.
\end{equation}
Since $D$ is dense we deduce from \eqref{e:keyprop6} that for $\mathscr{I}^m_1 $-a.e. $x_0 \in \mathbb{R}^n$ we have $\eta_{x_0} = f \, \mathcal{L}^l$ as measures. This gives the desired assertion.
\end{proof}

We are now in position to define the set $\hat{E}$ appearing in Theorem \ref{t:structure} for our relevant class of integralgeometric measures.

\begin{definition}[The set $\hat{E}$]
\label{d:defhatE}
    Assume that $\mathscr{I}^m_1$ is a finite integralgeometric measure. Then we define the set $\hat{E} \subset \mathbb{R}^n \times \Lambda$ as
    \begin{equation}
        \hat{E}:= \{(x,\lambda) \in \mathbb{R}^n \times \Lambda \ | \  \tilde{\eta}^\lambda_{P_\lambda(x)}(\{x\}) >0 \},
    \end{equation}
    where $\tilde{\eta}^\lambda_y$ is the good representative in the disintegration of $\mu_\lambda$ with respect to the map $P_\lambda \colon \mathbb{R}^n \to \mathbb{R}^m$ introduced in Definition \ref{d:goodrepdis}.
\end{definition}

\begin{proposition}
    The set $\hat{E}$ introduced above satisfies that $\hat{E}_\lambda$ is a Borel set in $\mathbb{R}^n$ for every $\lambda \in \Lambda$.
\end{proposition}

\begin{proof}
    Fix $\lambda \in \Lambda$. To simplify the notation we drop the dependence from $\lambda$ and simply write $\tilde{\eta}_y$ in place of $\tilde{\eta}^\lambda_y$. 
    
    We claim that the function $f \colon \mathbb{R}^m \times  \mathbb{R}^n \to [0,\infty)$ defined as $f(y,x):= \tilde{\eta}_y(\{x\})$ is Borel measurable for every $\lambda \in \Lambda$. To prove the claim, we start by showing that, for every $\varphi \in C^0_c(\mathbb{R}^n)$, the function $f_{\varphi,r} \colon \mathbb{R}^m \to [0,\infty)$ defined as
    \[
    f_{\varphi,r}(y) := 
    \begin{cases}
    \frac{1}{P_{\lambda \sharp}\mu(B_r(y))} \int_{P_\lambda^{-1}(B_r(y))} \varphi \, d\mu \ &\text{ if $y \in \text{supp}(P_{\lambda \sharp}\mu)$} \\
    0 & \text{ otherwise},
    \end{cases}
    \]
   is Borel measurable. Indeed, by choosing a family of mollifiers $(\phi_\epsilon)_{\epsilon >0}$, the function $f_{\varphi,r,\epsilon} \colon \mathbb{R}^m \to [0,\infty)$ defined as
    \[
    f_{\varphi,r,\epsilon}(y) := 
    \begin{cases}
    \frac{1}{P_{\lambda \sharp}(\mu * \phi_\epsilon)(B_r(y))} \int_{P_\lambda^{-1}(B_r(y))} \varphi \, d(\mu * \phi_\epsilon),  &\text{ if $y \in \text{supp}(P_{\lambda \sharp}\mu)$} \\
    0 & \text{ otherwise},
    \end{cases}
    \]
    is Borel measurable. This is because the map $y \mapsto \int_{P_\lambda^{-1}(B_r(y))} \varphi \, d(\mu * \phi_\epsilon)$ is continuous on the open set $\{y \ | \ P_{\lambda \sharp}(\mu *\phi_\epsilon)(B_r(y)) >0\}$ due to the fact that $\mu * \phi_\epsilon \ll \mathcal{L}^n$. Therefore, being $\text{supp}(P_{\lambda \sharp}\mu) \subset \{y \ | \ P_{\lambda \sharp}(\mu *\phi_\epsilon)(B_r(y)) >0\}$, and being the support of a Radon measure a Borel measurable set, we immediately deduce the Borel measurability of $f_{\varphi,r}$.  
    
    Since $\mu * \phi_\epsilon(U) \to \mu(U)$ as $\epsilon \to 0^+$ for any open set $U \subset \mathbb{R}^n$, we immediately infer that $f_{\varphi,r,\epsilon} \to f_{\varphi,r}$ as $\epsilon \to 0^+$ pointwise everywhere in $\mathbb{R}^m$. As a consequence, that $f_{\varphi,r}$ is a Borel measurable function. 
    
    From the very definition of $\tilde{\eta}_y$, we have that, defining the Borel set (Borel since both the liminf and the limsup below can be computed by restricting $r \in (0,\infty) \cap \mathbb{Q}$ due to left continuity of the map $r \mapsto  \frac{1}{P_{\lambda \sharp}\mu(B_r(y))} \int_{P_\lambda^{-1}(B_r(y))} \varphi \, d\mu$)
    \[
    Y_\varphi:= \{y \in \text{supp}(P_{\lambda \sharp}\mu) \ | \ \liminf_r f_{\varphi,r}(y) = \limsup_r f_{\varphi,r}(y)  \}
    \]
    then we have
    \[
    \lim_{r \to 0+} f_{\varphi,r}(y)=\int_{\mathbb{R}^n} \varphi \, d\tilde{\eta}_y, \ \ \text{ for every $y \in Y_\varphi$}.
    \]
    This tells us that the function $f_{\varphi} \colon \mathbb{R}^m \to [0,\infty)$ defined as
    \[
    f_{\varphi}(y):=
    \begin{cases}
    \int_{\mathbb{R}^n} \varphi \, d\tilde{\eta}_y, \ \text{ if $y \in Y_\varphi$}\\
    0, \ \text{ otherwise},
    \end{cases}
    \]
    is Borel measurable. 

    Choose $\varphi \in C^0_c(\mathbb{R}^n)$ with $\text{supp}(\varphi) \subset U_1(0)$ and $\varphi(0)=1$ and define $\varphi_r(z):= \varphi(z/r)$ for every $r>0$. Notice that, the set 
    \[
    D_\lambda:=\bigg\{y \in \text{supp}(P_{\lambda \sharp}\mu) \ \bigg| \ \lim_{r \to 0^+} \frac{1}{P_{\lambda \sharp}\mu(B_r(y))} \mu \restr P_{\lambda}^{-1}(B_r(y)) \text{ exists in the weak sense of measures} \bigg\}
    \]
    is contained in $Y_{\varphi_r}$ for every $r >0$. We further observe that, by considering a countable and dense subset $\mathcal{D} \subset C^0_c(\mathbb{R}^n)$, it is not difficult to verify the equality
    \[
    D_\lambda = \bigcap_{\varphi \in \mathcal{D}} Y_\varphi,
    \]
    from which we deduce that the set $D_\lambda$ is Borel. Therefore, the function $\mathbbm{1}_{D_\lambda}f_\varphi$ is Borel measurable for every $\varphi \in C^0_c(\mathbb{R}^n)$.

    Now define $f_r(y,x) \colon \mathbb{R}^m \times \mathbb{R}^n \to [0,\infty)$ as
    \[
    f_r(y,x) := \mathbbm{1}_{D_\lambda}(y)f_{\varphi_r \circ \tau_x}(y),
    \]
    where $\tau_x \colon \mathbb{R}^n \to \mathbb{R}^n$ is defined as $\tau_x(z):= z+x$. We already know that for every $x \in \mathbb{R}^n$ the map $f_r(\cdot,x)$ is Borel measurable. Moreover, it is not difficult to verify that, since $\lim_{x \to x'} \varphi_r \circ \tau_x(z) = \varphi_r \circ \tau_{x'}(z)$ pointwise for every $z \in \mathbb{R}^n$, then, for every $y \in \mathbb{R}^m$ the map $f_r(y,\cdot)$ is continuous. Hence $f_r$ is a Caratheodory function, and hence it is Borel jointly measurable on the product space $\mathbb{R}^m \times \mathbb{R}^n $. Finally, by exploiting that for every $x \in \mathbb{R}^n$ it holds $\lim_{r \to 0^+}\varphi_r \circ \tau_x(z)=  \mathbbm{1}_{\{x\}}(z)$ pointwise for every $z \in \mathbb{R}^n$, we infer from the Lebesgue's dominated convergence theorem that 
\[
\lim_{r \to 0^+} \int_{\mathbb{R}^n}\varphi_r \circ \tau_x \, d\tilde{\eta}_y = \tilde{\eta}_y(\{x\}), \ \ \text{ for every $(y,x) \in D_\lambda \times \mathbb{R}^n$}.
\]
Since by definition we have $\tilde{\eta}_y =0$ for every $y \in \mathbb{R}^m, \setminus D_\lambda$, we finally infer that 
\[
\lim_{r \to 0^+} f_r(y,x) = f(x,y) = \tilde{\eta}_y(\{x\}), \ \ \text{ for every $(y,x) \in \mathbb{R}^m \times \mathbb{R}^n$}, 
\]
from which our desired claim immediately follows. 

To conclude, we simply observe that $\tilde{\eta}_{P_\lambda(x)}(\{x\})$ coincides with the map $f(P_\lambda(x),x)$, which is Borel measurable as it is the composition of two Borel measurable maps.
\end{proof}

\begin{proposition}
\label{p:cor}
Let $(\mu_\lambda)_{\lambda \in \Lambda}$ be a measurable family of Borel regular measures in $\mathbb{R}^n$ and let $\Phi$ be a Orlicz function with superlinear growth at infinity. Assume that the measure $\mathscr{I}_\Phi^m$ constructed as in \eqref{e:caratheodoryc2} is finite. Moreover, assume that the measures $\tilde{\eta}^\lambda_y$ given by Definition \ref{d:goodrepdis} in the disintegration
\[
\mu_\lambda = \tilde{\eta}^\lambda_y \otimes P_{\lambda\sharp}\mu_\lambda,
\]
are $0$-rectifiable for $\mathcal{L}^l$-a.e. $\lambda \in \Lambda$ and for $\mathcal{L}^m$-a.e. $y \in \mathbb{R}^m$. 

Then, if $S \subset \mathbb{R}^n \times \Lambda$ is $(\mathscr{I}_{\Phi}^m \otimes \mathcal{L}^l)$-negligible, it holds 
\begin{equation}
\label{e:cor1}
    P_{\lambda\sharp}\mu_\lambda(P_\lambda(S_\lambda \cap \hat{E}_\lambda))=0, \ \ \text{ for $\mathcal{L}^l$-a.e. $\lambda \in \Lambda$}.
\end{equation}
where $\hat{E} \subset \Lambda \times \mathbb{R}^n$ is the set associated to the family $(\mu_\lambda)$ introduced in Definition \ref{d:defhatE}
\end{proposition}
\begin{proof}
Since $(\mathscr{I}_{\Phi}^m \otimes \mathcal{L}^l)$ is Borel regular we can reduce ourselves to prove the proposition in the case $S$ is Borel. By our hypothesis we can apply Tonelli's theorem to deduce that
\begin{equation}
\label{e:cor2}
\mathcal{L}^l(S_x)=0, \ \ \ \   \mathscr{I}_\Phi^m \text{-a.e. }x \in \mathbb{R}^n.
\end{equation}
Since $\mathscr{I}_\Phi^m(\mathbb{R}^n)<\infty$ implies $\hat{\mathscr{I}}_m(\mathbb{R}^n \times \Lambda) < \infty$ we can make use of disintegration theorem as in the proof of Proposition \ref{p:keyprop} to find a family of probability measures on $\Lambda$, say $(\eta_x)_{x \in \mathbb{R}^n}$, such that 
\begin{equation}
\label{e:cor3}
\hat{\mathscr{I}}_m  = \eta_x \otimes \mathscr{I}^m_1.
\end{equation}
 From Proposition \ref{p:keyprop} we already know that $\eta_x \ll \mathcal{L}^l$ for $\mathscr{I}^m_1$-a.e. $x \in \mathbb{R}^n$. Hence, \eqref{e:cor2} gives $\hat{\mathscr{I}}_m(S )=0$. Therefore, Remark \ref{p:nullset} allows us write
 \begin{equation}
 \label{e:cor3.1}
     \mu_\lambda(S_\lambda ) =0, \ \ \text{ for $\mathcal{L}^l$-a.e. }\lambda \in \Lambda.
 \end{equation}
The desired result will follow if we show that condition \eqref{e:cor3.1} implies \eqref{e:cor1}. To this purpose, fix $\lambda \in \Lambda$, and assume by contradiction that \eqref{e:cor3.1} holds while \eqref{e:cor1} is not satisfied. Thanks to the Borel regularity of the measure $\mu_\lambda$, we can find a Borel set $B \subset \mathbb{R}^n$ such that $S_\lambda \cap \hat{E}_\lambda \subset B$ and 
\[
\mu_\lambda(B) = \mu_\lambda(S_\lambda \cap \hat{E}_\lambda)  \leq \mu_\lambda(S_\lambda)=0.
\]

Notice that, since the projected sets $P_\lambda(B)$ and $P_\lambda(\hat{E}_\lambda \cap B)$ are $P_{\lambda \sharp} \mu_\lambda$-measurable (recall that $B$ and $\hat{E}_\lambda$ are Borel sets, so we can apply the measurable projection theorem \cite[Proposition 8.4.4]{coh} to conclude that their projections are universally measurable, and hence $P_{\lambda \sharp} \mu_\lambda$-measurable because $\mu_\lambda$ is Borel and finite), and they satisfy

\[
P_{\lambda \sharp} \mu_\lambda(P_\lambda(B \cap \hat{E}_\lambda)) \geq  P_{\lambda \sharp} \mu_\lambda(P_\lambda(S_\lambda \cap \hat{E}_\lambda)) > 0.
\]
Clearly, since $\tilde{\eta}^\lambda_y$ is $0$-rectifiable, we have, by construction of the set $\hat{E}_\lambda$, the representation
\[
\tilde{\eta}^\lambda_y = \sum_{x \in \hat{E}_\lambda} c_x^\lambda \big(\delta_{x} \restr P^{-1}_\lambda(y)\big),
\]
for some \emph{strictly positive} constants $c^\lambda_x > 0$. This implies that the following implication holds:
\begin{equation}
\label{e:htc}
y \in P_\lambda(\hat{E}_\lambda \cap B) \ \Rightarrow\ \tilde{\eta}^\lambda_y(\hat{E}_\lambda \cap B) > 0, \quad \text{for $P_{\lambda \sharp} \mu_\lambda$-a.e.\ $y \in \mathbb{R}^m$}.
\end{equation}
We can thus write

\begin{align*}
    \mu_\lambda(B) = \int_{P_\lambda(B)} \tilde{\eta}^\lambda_y(B) \, P_{\lambda \sharp} \mu_\lambda(y) &= \int_{\{y \in P_\lambda(B) \ | \ \tilde{\eta}^\lambda_y \text{ is non-null and $0$-rectifiable} \}} \tilde{\eta}^\lambda_y(B) \, P_{\lambda \sharp} \mu_\lambda(y) \\
    &= \int_{\{y \in P_\lambda(\hat{E}_\lambda \cap B) \ | \ \tilde{\eta}^\lambda_y \text{ is non-null and $0$-rectifiable} \}} \tilde{\eta}^\lambda_y(\hat{E}_\lambda \cap B) \, P_{\lambda \sharp} \mu_\lambda(y).
\end{align*}
Using the implication~\eqref{e:htc}, we deduce that the last integral in the above formula must be strictly positive, thus contradicting the assumption $\mu_\lambda(B) = 0$.  
This concludes the proof.

\end{proof}

We are now in position to prove the central result of this paper. 

\begin{theorem}
\label{t:central}
Let $(\mu_\lambda)_{\lambda \in \Lambda}$ be a measurable family of Borel regular measures in $\mathbb{R}^n$ and let $\Phi$ be a Orlicz function with superlinear growth at infinity. Assume that the measure $\mathscr{I}_\Phi^m$ constructed as in \eqref{e:caratheodoryc2} is finite and integralgeometric. Then 
\[
\mathscr{I}_\Phi^m(\mathbb{R}^n \setminus R)=0,
\]
for some countably $m$-rectifiable Borel set $R \subset \mathbb{R}^n$. If in addition there exists $\alpha \in (0,1]$ such that $P_\lambda \colon \mathbb{R}^n \to \mathbb{R}^m$ is $\alpha$-H\"older for $\mathcal{L}^l$-a.e. $\lambda \in \Lambda$, then we have also $\mathscr{I}^m_\Phi \ll \mathcal{H}^{\alpha m} \restr R$.
\end{theorem}

\begin{proof}
We claim that the Borel set $E$ and the measure $\mathscr{I}_\Phi^m$ satisfy the hypotheses of Theorem \ref{t:structure}, with $\sigma_\lambda$ replaced by $P_{\lambda\sharp}\mu_\lambda$. In fact, it suffices to verify condition \eqref{e:structure1}. This will follow from Proposition \ref{p:cor} once we verify that the family $(\mu_\lambda)$, which generates the integral geometric measure $\mathscr{I}_\Phi^m$, satisfies all the hypothesis. To this regard, we need only to check that the measures $\tilde{\eta}^\lambda_y$ are $0$-rectifiable for $\mathcal{L}^l$-a.e. $\lambda \in \Lambda$ and for $P_{\lambda \sharp}\mu_\lambda$-a.e. $y \in \mathbb{R}^m$. But this directly follows by applying conditions \eqref{e:condigm1}-\eqref{e:condigm6} to the disintegration of $\mu_\lambda$ with respect to $P_\lambda$. 

We find therefore a countably $(\mathscr{I}_\Phi^m ,m)$-rectifiable and $\mathscr{I}_\Phi^m$-measurable set $R'$ such that $\mathbb{R}^n \setminus R'$ is a $\sigma$-compact purely $(\mathscr{I}_\Phi^m ,m)$-unrectifiable set and satisfies \eqref{e:structure3}, namely for $\mathcal{L}^l$-a.e. $\lambda \in \Lambda$ we have
\begin{equation}
    P_{\lambda\sharp}\mu_\lambda(P_\lambda( \hat{E}_\lambda  \setminus R'))=0,
\end{equation}
where the set $\hat{E} \subset \Lambda \times \mathbb{R}^n$ has been introduced in Definition \ref{d:defhatE}. Moreover, since $\mu_\lambda$ disintegrates atomically with respect to $P_\lambda$, by arguing as in the proof of the above proposition, we infer that $\mu_\lambda = \mu_\lambda \restr \hat{E}_\lambda$ for almost every $\lambda \in \Lambda$.  
This in turn immediately implies via property \eqref{e:zeroiff} that $\mathscr{I}_1^m(\mathbb{R}^n \setminus R') = 0$. From Proposition~\ref{p:borelreg}, we also have $\mathscr{I}_\Phi^m(\mathbb{R}^n \setminus R') = 0$.  
As a consequence, the first part of the theorem follows by the Borel regularity of $\mathscr{I}^m_\Phi$.

To prove the remaining part we observe that 
\[
\mathcal{H}^m(P_\lambda(B)) \leq \sup_{x \neq x'} \bigg(\frac{P_\lambda(x)-P_\lambda(x')}{|x-x'|^\alpha} \bigg)^m\mathcal{H}^{\alpha m}(B), \ \ \text{ for every $B \subset \mathbb{R}^n$ Borel},
\]
and arguing as in \cite[Theorem 2.10.10, Corollary 2.10.11]{fed1} we infer
 \begin{equation}
 \label{e:op100}
 \int_{\mathbb{R}^m} \mathcal{H}^0(R \cap B \cap P^{-1}_\lambda(y)) \, dy \leq \sup_{x \neq x'} \bigg(\frac{P_\lambda(x)-P_\lambda(x')}{|x-x'|^\alpha} \bigg)^m \,  \mathcal{H}^{\alpha m}(R \cap B), \ \ B \subset \mathbb{R}^n \text{ Borel}.
 \end{equation}
Since $\eta^\lambda_y $ is concentrated on $P^{-1}_\lambda(y)$ for $P_{\lambda\sharp}\mu_\lambda$-a.e. $y \in \mathbb{R}^m$, we deduce that $\eta^\lambda_y \restr R$ is concentrated on $R \cap P^{-1}_\lambda(y)$ for $P_{\lambda\sharp}\mu_\lambda$-a.e. $y \in \mathbb{R}^m$. From \eqref{e:op100} we can thus write for $\mathcal{L}^l$-a.e. $\lambda \in \Lambda$
 \[
 \mu_\lambda \restr R = \eta^\lambda_y \restr (R \cap P^{-1}_\lambda(y))  \otimes P_{\lambda\sharp}\mu_\lambda \ll \mathcal{H}^0 \restr (R \cap P^{-1}_\lambda(y))  \otimes P_{\lambda\sharp}\mu_\lambda \ll \mathcal{H}^{\alpha m} \restr R
 \]
 where we used the hypothesis $P_{\lambda\sharp}\mu_\lambda \ll \mathcal{L}^m$. Finally, since $\mathscr{I}^m_1(E \setminus R)=0$, we can make use of formula \eqref{e:equal3} together with $\mu_\lambda \restr R \ll \mathcal{H}^{\alpha m} \restr R$ to infer $\mathscr{I}^m_\Phi \ll \mathscr{I}^m_1 \ll \mathcal{H}^{\alpha m} \restr R$.
 \end{proof}

\section{Structure of Federer's Integralgeometric measure}
\label{s:application}

In this section, we provide a positive answer to problem (Q.1) for every $p >1$. Throughout the remainder of the section, we restrict our attention to Orlicz functions of the form $\Phi(t)= t^p$ for some exponent $p \geq 1$, and we consider as our transversal family the set of orthogonal projections.

For the reader's convenience, we recall the relevant notation. To properly address the problem, we introduce the exact renormalization constant $\beta_p(n,m)$ used in the construction of the Integralgeometric measure $\mathcal{I}^m_p$. For a given exponent $p \in [1, \infty]$, the constant $\beta_p(n,m)$ is defined as the one satisfying (see \cite[Subsection 2.7.16)(6)]{fed1})
\begin{equation}
\beta_p(n,m)= \bigg(\int_{\text{Gr}(n,m)} |\langle d\pi_V, \tau \rangle|^p \, d\gamma_{n,m}(V) \bigg)^{\frac{1}{p}}
\end{equation}
for every unitary simple $m$-vector $\tau$ of $\mathbb{R}^n$ and where $\langle \cdot, \cdot \rangle$ denotes the duality pairing between $m$-covector and $m$-vector of $\mathbb{R}^n$.  For every $V \in \text{Gr}(n,m)$ and every Borel set $B \subset \mathbb{R}^n$ we set
\begin{equation}
    \label{e:fedigm1}
    g_B(V):= \mathcal{H}^m(\pi_V(B)),
\end{equation}
and the set function
\begin{equation}
    \label{e:fedigm2}
    \eta_p(B):=\frac{\|g_B\|_{L^p(\text{Gr}(n,m))}}{\beta_p(n,m)}.
\end{equation}
The $m$-dimensional Integralgeometric measure in $\mathbb{R}^n$ with exponent $p \in [0,\infty]$, denoted by $\mathcal{I}^m_p$, is then defined via Caratheodory's construction as in \eqref{e:caratheodoryc2} via the Gauge function $\eta_p$.

We observe that, in order to show $\mathcal{I}^m_{p_1} = \mathcal{I}^m_{p_2}$ for every $1 < p_1 \leq p_2 \leq \infty$, it is sufficient to prove that, for every $p > 1$ and every Borel set $E \subset \mathbb{R}^n$ with finite $\mathcal{I}^m_p$-measure, there exists a countably $m$-rectifiable Borel set $R \subset \mathbb{R}^n$ such that
\begin{equation}
\label{e:mainfed}
\mathcal{I}^m_p(E \setminus R) = 0.
\end{equation}
Indeed, if for a Borel set $E \subset \mathbb{R}^n$ at least one of the quantities $\mathcal{I}^m_{p_1}(E)$ or $\mathcal{I}^m_{p_2}(E)$ is finite, then there exists a countably $m$-rectifiable Borel set $R \subset \mathbb{R}^n$ such that either $\mathcal{I}^m_{p_1}(E \setminus R) = 0$ or $\mathcal{I}^m_{p_2}(E \setminus R) = 0$. Since the measures $\mathcal{I}^m_p$ share the same family of negligible sets (this follows immediately from Lemma~\ref{l:fedigmt} below together with condition \eqref{e:samenegset}), we deduce that both conditions
\[
\mathcal{I}^m_{p_1}(E \setminus R) = 0 \quad \text{and} \quad \mathcal{I}^m_{p_2}(E \setminus R) = 0
\]
hold. Moreover, by observing that the renormalization constant $\beta_p(n,m)$ is defined precisely so that
\[
\mathcal{I}^m_{p_1} \restr W = \mathcal{I}^m_{p_2} \restr W \quad \text{for every affine $m$-plane } W \subset \mathbb{R}^n,
\]
it is not difficult to show that $\mathcal{I}^m_{p_1}$ and $\mathcal{I}^m_{p_2}$ coincide with $\mathcal{H}^m$ on every countably $m$-rectifiable Borel subset of $\mathbb{R}^n$ (see \cite[Theorem~3.3.13]{fed1}):
\begin{proposition}
\label{p:coihmim}
    Let $R \subset \mathbb{R}^n$ be a countably $m$-rectifiable Borel set. Then
    \[
    \mathcal{I}^m_{p}(B) = \mathcal{H}^m(B), \ \ \text{ for every Borel set $B \subset R$},
    \]
    whenever $1 \leq p < \infty$.
\end{proposition}
Hence, the identity $\mathcal{I}^m_{p_1} = \mathcal{I}^m_{p_2}$ follows directly from condition~\eqref{e:mainfed} via Proposition~\ref{p:coihmim}.

\vspace{3mm}

 In order to prove condition~\eqref{e:mainfed}, we rewrite $\mathcal{I}^m_p$ as an integralgeometric measure $\mathscr{I}^m_p$ in the sense of Definition~\ref{d:igm}. This reformulation does not follow directly from the definition, due to the fact that the map $B \mapsto g_V(B)$ is not a measure.

 In order to overcome this issue we proceed as follows. Define for every $V \in \text{Gr}(n,m)$ and for every $B \subset \mathbb{R}^n$ Borel, the Borel measure $\mu_V$ on $\mathbb{R}^n$ by
\begin{equation}
\label{e:fedigm3.1}
\mu_V(B):= \int_{\mathbb{R}^m} \mathcal{H}^0(B \cap \pi_V^{-1}(y)) \, dy.
\end{equation}
The required measurability to define the integral in \eqref{e:fedigm3.1} can be found in \cite[Subsection 2.10.16]{fed1}. By letting $f_B(V):= \mu_V(B)$ for every $V \in \text{Gr}(n,m)$ and $B \subset \mathbb{R}^n$ Borel we define the set function 
\begin{equation}
\label{e:fedigm4.1}
    \zeta_p(B):= \|f_B\|_{L^p(\text{Gr}(n,m))}.
\end{equation}
Consider then the Borel regular measure $\mathscr{I}_p^m$ on $\mathbb{R}^n$ (see Proposition \ref{p:borelreg}) defined as in \eqref{e:caratheodoryc2} for every $E \subset \mathbb{R}^n$ by
\begin{equation}
    \label{e:fedigm6}
    \mathscr{I}_p^m(E):=\sup_{\delta>0} \, \inf \big\{ \sum_{B \in G_\delta} \zeta_p(B) \ | \ G_\delta \text{ countable Borel cover of $A$, }  \text{diam}(B)\leq \delta \big\}.
\end{equation}

We have the following lemma.

\begin{lemma}
\label{l:fedigmt}
For every $p \geq 1$ we have $\mathcal{I}^m_p = \beta_p(n,m)\mathscr{I}_p^m$.
\end{lemma}
\begin{proof}
In order to simplify the notation, we assume that the value of the constant $\beta_p(n,m)$ is $1$.

Being both measures $\mathcal{I}^m_p$ and $\mathscr{I}_p^m$ Borel regular it is enough to prove that they coincide on Borel sets. In addition, since $\eta_p(B) \leq \zeta_p(B)$ for every $B \subset \mathbb{R}^n$ Borel we have only to prove that 
\begin{equation}
\label{e:fedigm8}
    \mathcal{I}^m_p(B) \geq \mathscr{I}_p^m(B), \ \ B \subset \mathbb{R}^n \text{ Borel}.
\end{equation}
For this purpose fix $B \subset \mathbb{R}^n$ Borel and assume  without loss of generality that $\mathcal{I}^m_p(B) < \infty$. Consider for $k=1,2,\dotsc$ a sequence of Borel countable coverings of $B$, say $(D_k)$, such that $D \in D_k$ implies $\text{diam}(D) \leq 1/k$ and $ 
\lim_{k \to \infty} \sum_{D \in  D_k} \eta_p(D) = \mathcal{I}^m_p(B)$. Notice that for every $k$ 
\[
\begin{split}
\sum_{D \in D_k} \eta_p(D) &= \sum_{D \in D_k}\bigg(\int_{\text{Gr}(n,m)} \mathcal{H}^m(\pi_V(D))^p \, d\gamma_{n,m}\bigg)^{\frac{1}{p}} \\
&\geq \bigg(\int_{\text{Gr}(n,m)} \big[\sum_{D \in D_k} \mathcal{H}^m(\pi_V(D))\big]^p \, d\gamma_{n,m}\bigg)^{\frac{1}{p}} \\
&= \bigg(\int_{\text{Gr}(n,m)}\bigg(\int_{V} \mathcal{H}^0(\{D \in D_k \ | \ \pi_V^{-1}(y) \cap D \neq \emptyset\}) \,d\mathcal{H}^m(y)\bigg)^p d\gamma_{n,m}\bigg)^{\frac{1}{p}}.
\end{split}
\]
Since every $D_k$ is made of sets whose diameter is shrinking to zero as $k \to \infty$, we have
\[
\liminf_{k \to \infty}\mathcal{H}^0(\{D \in D_k \ | \ \pi_V^{-1}(y) \cap D \neq \emptyset \}) \geq \mathcal{H}^0(\pi_V^{-1}(y) \cap B), 
\]
for every $V \in \text{Gr}(n,m)$ and $y \in V$. Therefore we can apply Fatou's lemma to deduce
\begin{equation}
\label{e:fedigm7}
\begin{split}
&\mathcal{I}^m_p(B) = \liminf_{k \to \infty} \sum_{D \in D_k} \eta_p(D) \\
&\geq \bigg(\int_{\text{Gr}(n,m)}\bigg(\int_{V} \liminf_{k \to \infty}\mathcal{H}^0(\{D \in D_k \ | \ \pi_V^{-1}(y) \cap D \neq \emptyset\}) \,d\mathcal{H}^m(y)\bigg)^p d\gamma_{n,m}\bigg)^{\frac{1}{p}}\\
&\geq \bigg(\int_{\text{Gr}(n,m)}\bigg(\int_{\mathbb{R}^m} \mathcal{H}^0(\pi_V^{-1}(y) \cap B) \,d\mathcal{H}^m(y)\bigg)^p d\gamma_{n,m}\bigg)^{\frac{1}{p}} = \zeta_p( B).
\end{split}
\end{equation}
For every $\delta>0$, since $\mathcal{I}^m_p$ is a Borel measure, we easily find a countable Borel covering of $B$, say $\tilde{G}_\delta$, with $\text{diam}(\tilde{B}) \leq \delta$ for $\tilde{B} \in \tilde{G}_\delta$ and $\sum_{\tilde{B} \in \tilde{G}_\delta} \mathcal{I}^m_p(\tilde{B})=\mathcal{I}^m_p(B)$. Thanks to \eqref{e:fedigm7} we can write for every $\delta>0$
\[
\begin{split}
\mathcal{I}^m_p(B)=\sum_{\tilde{B} \in \tilde{G}_\delta}\mathcal{I}^m_p(\tilde{B}) \geq \sum_{\tilde{B} \in \tilde{G}_\delta}\zeta_p(\tilde{B})
&\geq \inf_{G_\delta} \sum_{\tilde{B} \in G_\delta}\zeta_p(\tilde{B}) \geq \mathscr{I}_p^m(B) - o(1).
\end{split}
\]
where $G_\delta$ denotes any countable Borel coverings of $B$ made of sets having diameter less or equal than $\delta$ and $o(1) \to 0$ as $\delta \to 0^+$.
Taking the limit as $\delta \to 0^+$ on both sides of the previous inequality we obtain \eqref{e:fedigm8}.
\end{proof}

We are now in position to prove the rectifiability property of $\mathcal{I}^m_p$.

\begin{proof}[Proof of Theorem \ref{t:fedpronew}]
Thanks to Lemma \ref{l:fedigmt} we know that the two measures $\beta_p(n,m)\mathcal{I}^m_p \restr E$ and $\mathscr{I}^m_p \restr E$ coincide. Eventually, since $\mathscr{I}^m_p \restr E$ is a finite integralgeometric measure in the sense of Definition \ref{d:igm}, we obtain the desired conclusion by applying Theorem \ref{t:central} together with Remark \ref{r:tratopro}.
\end{proof}

We conclude this section with a Remark.

\begin{remark}
    \label{r:intgeophi}
    Given a Orlicz function $\Phi$, Theorem~\ref{t:fedpronew} extends to the measure $\mathcal{I}^m_\Phi$, provided that $\Phi$ has superlinear growth at infinity. Moreover, Proposition \ref{p:coihmim} remains valid under the same assumption, upon appropriately choosing the renormalization constant
\[
\beta_\Phi(n,m) := \|g_\tau\|_\Phi, \qquad \text{where } g_\tau \colon \mathrm{Gr}(n,m) \to [0,1], \quad g_\tau(V) := |\langle d\pi_V, \tau \rangle|,
\]
for every simple $m$-vector $\tau$. In particular, problem (Q.1) admits a positive solution for every Orlicz function $\Phi$ with superlinear growth at infinity

\end{remark}

\section{Vitushkin's conjecture }
\label{s:vitushkin}
In this section we show how our integral geometric method allows one to make progress in the understanding of Vitushkin's conjecture. Let $m,n$ be two strictly positive integers with the following order $m < n$. For reader's convenience we recall here the relevant notation. The Favard length of a Borel set $B \subset \mathbb{R}^n$ is defined as
\begin{equation}
    \text{Fav}(B) = \int_{\text{Gr}(n,m)} \mathcal{H}^m(\pi_V(B)) \, d\gamma_{n,m}(V).
\end{equation}
 For a given Orlicz function $\Phi$ we introduce the $\Phi$-Favard length of a Borel set $B \subset \mathbb{R}^n$, denoted by $\text{Fav}_\Phi(B)$, as
\begin{equation}
    g_B(V):= \mathcal{H}^m(\pi_V(B)) \qquad \text{Fav}_\Phi(B):= \|g_B\|_\Phi.
\end{equation}
When $\Phi(t) = t^p$ for some exponent $p \geq 1$ we compactly denote $\text{Fav}_\Phi$ by $\text{Fav}_p$. Notice that, with this notation we have $\text{Fav}_1 \equiv \text{Fav}$.

Following Ahlfors \cite{ahl}, we further recall that, for a compact set $K \subset \mathbb{C}$, the analytic capacity of $K$, denoted by $\gamma(K)$, can be defined as 
\begin{equation}
   \gamma(K)= \sup |f'(\infty)|,
\end{equation}
where the supremum is taken over all analytic functions $f \colon \mathbb{C} \setminus K \to \mathbb{C}$ with $|f| \leq 1$ on $\mathbb{C} \setminus K$, and $f'(\infty)= \lim_{z \to \infty} z(f(z) -f(\infty))$.

The main result providing a sufficient geometric condition for $\gamma(K)>0$ which we rely on is due to Calder\'on \cite{cal} and reads as follows. 

\begin{theorem}[Calder\'on (1977)]
\label{t:cald}
    Let $K \subset \mathbb{C}$ be a compact set. If $\mathcal{H}^1(K \cap \emph{Im}(\varphi)) >0$ for some Lipschitz curve $\varphi \colon \mathbb{R} \to \mathbb{C}$, then $\gamma(K)>0$. 
\end{theorem}

The above result allows us to transform the problem of proving that a set has positive analytic capacity into a rectifiability problem: to prove that $K$ satisfies $\gamma(K) > 0$, we reduce the task to constructing a countably $1$-rectifiable Borel set $R \subset \mathbb{C}$ such that $\mathcal{H}^1(K \cap R) > 0$.

\subsection{A single-scale result}
\label{ss:ssresult}

In this subsection, we show how our central result (Theorem~\ref{t:central}) can be applied to obtain new insights related to Vitushkin's conjecture, specifically Corollary~\ref{c:singlescale}. To this end, we extend the Besicovitch–Federer projection theorem to sets $E \subset \mathbb{R}^n$ that intersect a typical affine $(n-m)$-plane in only finitely many points. The desired result concerning Vitushkin's conjecture then follows as a corollary, by virtue of Calder\'on's theorem.

\begin{proof}[Proof of Theorem \ref{t:besfedintro}]
Recall from Remark~\ref{r:tratopro} that the family of orthogonal projections $(\pi_V)_{V \in \mathrm{Gr}(n,m)}$ can be seen as a union of finitely many transversal family of maps by localizing in $\mathrm{Gr}(n,m)$. Hence, Theorem~\ref{t:central} applies to any integralgeometric measures in the sense of Definition \ref{d:igm} constructed via the family of orthogonal projections.

\vspace{3mm}

\underline{\emph{Step 1}}. Let $\mu$ be a finite measure in $\mathbb{R}^n$. We claim that $\mu$ can be decomposed as
\begin{equation}
    \mu(B)= \mu_r(B) + \mu_s(B), \ \ \text{ for every Borel set $B \subset \mathbb{R}^n$},
\end{equation}
where $\mu_r$ and $\mu_s$ are finite measures satisfying
\begin{itemize}
\item \text{If $F \subset \mathbb{R}^n$ intersects a typical affine $(n\!-\!m)$-plane in finitely many points, then $\mu_r \restr F$ is $m$-rectifiable}. 
\item  $ \pi_{V\sharp} \mu_s \, \bot \, \mathcal{H}^m$, \ \ \text{ for $\gamma_{n,m}$-a.e. $V \in \text{Gr}(n,m)$.}
\end{itemize}

To prove the claim, denote by $\mathcal{S}$ the family of (positive) Radon measures $\eta$ in $\mathbb{R}^n$ such that 
 \[
 |\mu -\eta|(\mathbb{R}^n) + \eta(\mathbb{R}^n)= \mu(\mathbb{R}^n) \quad \text{and} \quad \eta \restr F \text{ is $m$-rectifiable}
 \]
 whenever $F \subset \mathbb{R}^n$ satisfies the required finite-slicing condition and where $|\mu -\eta|$ denotes the total variation measure. Consider the minimum problem
\begin{equation}
\label{e:intros100}
\min_{\eta \in \mathcal{S}} |\mu -\eta|(\mathbb{R}^n). 
\end{equation}
Let $\eta$ be a (positive) Radon measure. Notice that, for every Borel set $B \subset \mathbb{R}^n$, we have
\[
|\mu - \eta|(B) + \eta(B) \geq \mu(B).
\]
Therefore, the condition $\eta \in \mathcal{S}$ implies that equality holds for all Borel sets $B \subset \mathbb{R}^n$:
\[
|\mu - \eta|(B) + \eta(B) = \mu(B).
\]
In particular, this implies that both $|\mu - \eta| \leq \mu$ and $\eta \leq \mu$. By the Radon--Nikodym Theorem, there exists a Borel function $f_\eta \colon \mathbb{R}^n \to [0,1]$ such that
\[
\eta = f_\eta \, \mu \quad \text{and} \quad |\mu - \eta| = (1 - f_\eta)\, \mu.
\]
 Now let $(\eta_i)$ be a minimizing sequence for \eqref{e:intros100} and let $f_{\eta_i} \colon \mathbb{R}^n \to [0,1]$ be such that $\eta_i= f_{\eta_i} \, \mu$ for $i=1,2,\dotsc$. Define $f_i \colon \mathbb{R}^n \to [0,1]$ as $f_i := \sup_{1 \leq j \leq i} f_{\eta_j}$ and notice that the Radon measure $\mu_i := f_i \, \mu$ satisfies
\[
|\mu - \mu_i|(B) + \mu_i(B)= \mu(B) \ \ \text{ and } \ \ |\mu - \mu_i|(\mathbb{R}^n) \leq |\mu - \eta_i|(\mathbb{R}^n), 
\]
whenever $B \subset \mathbb{R}^n$ is Borel and $i=1,2,\dotsc$.
Since $\mu_i \leq \sum_{1 \leq j \leq i} \eta_j$, and by assumption each $\eta_j \restr F$ is $m$-rectifiable, the Radon--Nikodym Theorem implies that $\mu_i \restr F$ is also $m$-rectifiable. Hence, the sequence $(\mu_i)$ remains a minimizing sequence for \eqref{e:intros100}.

Moreover, since $f_i \nearrow f$ for some Borel function $f \colon \mathbb{R}^n \to [0,1]$, we can apply the Dominated Convergence Theorem to deduce that $|\mu_r - \mu_i|(\mathbb{R}^n) \to 0$ as $i \to \infty$, where $\mu_r := f \, \mu$. Thanks to this strong convergence, we immediately obtain
\[
|\mu - \mu_r|(B) + \mu_r(B) = \mu(B), \quad \text{and} \quad |\mu - \mu_r|(B) = \lim_{i \to \infty} |\mu - \mu_i|(B),
\]
for every Borel set $B \subset \mathbb{R}^n$. Additionally, the $m$-rectifiability of $\mu_r \restr E$ follows directly. Therefore, $\mu_r$ is a minimizer for \eqref{e:intros100}.

Finally, if we define $\mu_s := \mu - \mu_r$, the fact that $\mu_r \in \mathcal{S}$ implies
\[
\mu(B) = \mu_r(B) + |\mu_s|(B) = \mu_r(B) + \mu_s(B), \quad \text{for every Borel set } B \subset \mathbb{R}^n.
\]

It remains only to prove that $\mu_s$ projects singularly for $\gamma_{n,m}$-almost every $V \in \mathrm{Gr}(n,m)$. We argue by contradiction: suppose that the set
\[
\Lambda := \{ V \in \mathrm{Gr}(n,m) : \pi_{V \sharp} \mu_s \text{ is not singular with respect to } \mathcal{H}^m \}
\]
satisfies $\gamma_{n,m}(\Lambda) > 0$.

The disintegration theorem gives for every $V \in \Lambda$ a family of Radon measures $(\eta^V_y)_{y \in V}$ such that
\begin{equation}
\label{e:intros3}
\mu_s(B) = \int_{V} \eta^V_y(B) \, d[\pi_{V\sharp}\mu_s]^a(y) + \int_{V} \eta^V_y(B) \, d[\pi_{V\sharp}\mu_s]^s(y), \ \ B \subset{\mathbb{R}^n} \text{ Borel},
\end{equation}
where $[\pi_{V\sharp}\mu_s]^a$ and $ [\pi_{V\sharp}\mu_s]^s$ denote the absolutely continuous part and the singular part of $\pi_{V\sharp}\mu_s$ with respect to $\mathcal{H}^m$, respectively. Notice that, by exploiting that the two measures in the right-hand side of the previous equality are mutually orthogonal, it can be proved that the maps $g^a_B \colon \text{Gr}(n,m) \to [0,\infty)$ and $g^s_B \colon \text{Gr}(n,m) \to [0,\infty)$ defined as 
\[
g^a_B(V):= \int_{V} \eta^V_y(B) \, d[\pi_{V\sharp}\mu_s]^a(y) \ \ \text{ and } \ \ g^s_B(V) := \int_{V} \eta^V_y(B) \, d[\pi_{V\sharp}\mu_s]^s(y), 
\]
are $\gamma_{n,n}$-measurable whenever $B \subset \mathbb{R}^n$ is Borel. Indeed, by exploiting the uniqueness property in the disintegration Theorem \ref{t:disthm}, it is not difficult to verify that for every Borel set $B \subset \mathbb{R}^n$, the following equality holds:
\[
\eta^V_y(B) = \theta_{(y,V)}(B \times \mathrm{Gr}(n,m)) \quad \text{for } \gamma_{n,m}\text{-a.e. } V \in \mathrm{Gr}(n,m) \text{ and } \pi_{V\sharp}\mu\text{-a.e. } y \in V,
\]
where the family of probability measures $(\theta_{(y,V)})$ arises from the disintegration of the product measure $\mu_s \otimes \gamma_{n,m}$ via the map $\hat{\pi}(x,V) := (\pi_V(x), V)$, for $(x, V) \in \mathbb{R}^n \times \mathrm{Gr}(n,m)$.

As a consequence, the map $(y,V) \mapsto \eta^V_y(B)$ is measurable with respect to $\hat{\pi}_\sharp(\mu_s \otimes \gamma_{n,m})$. Furthermore, since the following orthogonal decomposition holds:
\[
\hat{\pi}_\sharp(\mu_s \otimes \gamma_{n,m}) = \big([\pi_{V\sharp}\mu_s]^a \otimes \gamma_{n,m}\big) \oplus \big([\pi_{V\sharp}\mu_s]^s \otimes \gamma_{n,m}\big),
\]
it follows that the map $(y,V) \mapsto \eta^V_y(B)$ is measurable with respect to both $[\pi_{V\sharp}\mu_s]^a \otimes \gamma_{n,m}$ and $[\pi_{V\sharp}\mu_s]^s \otimes \gamma_{n,m}$. This, in turn, implies the $\gamma_{n,m}$-measurability of the maps $g^a_B$ and $g^s_B$.

Integrating both sides of \eqref{e:intros3} with respect to $\gamma_{n,m}$ we get two finite Radon measures in $\mathbb{R}^n$, say $\eta_a$ and $\eta_s$, such that
\begin{equation}
\label{e:intros299}
 \mu_s(B) = \eta_a(B) + \eta_s(B) \ \ B \subset{\mathbb{R}^n} \text{ Borel}.
\end{equation}

Thanks to the contradiction argument, we deduce that $\eta_a \neq 0$. By virtue of \eqref{e:intros299}, if we can show that $\eta_a \restr F$ is $m$-rectifiable for every Borel set $F \subset \mathbb{R}^n$ satisfying the finite-slicing condition mentioned above, then the minimality of $\mu_s$ would contradict the fact that $\eta_a \neq 0$.
 To this regard, we consider the family of Borel measures $(\mu_V)_{V \in \text{Gr}(n,m)}$, defined as 
\[
\mu_V(B):= g^a_{F \cap B}(V), \ \ B \subset \mathbb{R}^n \text{ Borel}, 
\]
 and the associated Borel regular measures $\mathscr{I}^m_\infty$ according to \eqref{e:caratheodoryc2}. We claim that $\mathscr{I}^m_\infty$ is a finite integralgeometric measure. Properties \eqref{e:condigm1}-\eqref{e:condigm6} follow by construction. The finiteness is a consequence of the following inequality 
\begin{equation}
\label{e:radonm1}
\begin{split}
\|g^a_{F \cap B}\|_{L^\infty(\text{Gr}(n,m))}=\|g^a_{F \cap B}\|_{L^\infty(\Lambda)}
&= \esup_{V \in \Lambda} \int_{V} \eta^V_y(F \cap B) \, d[\pi_{V\sharp}\mu_s]^a(y)\\
&\leq \mu_s(F \cap B)\\
&\leq \mu(F \cap B),
\end{split}
\end{equation}
whenever $B \subset \mathbb{R}^n$ is Borel. Indeed, given $\delta > 0$, we can consider a countable Borel covering $(B_i)$ of $\mathbb{R}^n$ consisting of sets with diameter at most $\delta$, and such that for every $i = 1, 2, \dotsc$, the number of overlapping sets satisfies
\[
\#\{ B_j \mid B_j \cap B_i \neq \emptyset \} \leq N,
\]
for some dimensional constant $N$. For instance, one can take the closed cubes of an $n$-dimensional $\delta$-grid aligned with the coordinate axes. Using \eqref{e:radonm1}, we then obtain
\[
\sum_{i=1} \|g^a_{F \cap B_i}\|_{L^\infty(\Lambda)} \leq N  \mu(\mathbb{R}^n) < \infty.
\]
This last inequality, together with the arbitrariness of $\delta$, implies the finiteness of $\mathscr{I}^m_\infty$. By Theorem~\ref{t:central}, we conclude that $\mathscr{I}^m_\infty$ is $m$-rectifiable. Moreover, we already know from Proposition~\ref{r:inequality} that $\mathscr{I}^m_1$ is also $m$-rectifiable. Finally, using formula~\eqref{e:equal3}, we infer that for every Borel set $B \subset \mathbb{R}^n$, the following holds

\begin{align*}
 \mathscr{I}^m_1(B) = \int_{\Lambda} \mu_{V}(F \cap B) \, d\gamma_{n,m}&=\int_{\Lambda}\bigg( \int_{V} \eta^V_y( F \cap B) \, d[\pi_{V\sharp}\mu_s]^a(y)\bigg)d\gamma_{n,m}\\
&= \eta_a \restr F(B).
\end{align*}
 This gives that $\eta^a \restr F$ is $m$-rectifiable and the desired contradiction follows. The claim is proved.

\vspace{4mm}

\underline{\emph{Step 2}}. To conclude the proof we first construct an integral measure $\tilde{\mathscr{I}}^m_1$ on $E$ satisfying the following condition for every countably $m$-rectifiable set $R \subset \mathbb{R}^n$:
\begin{equation}
\label{e:implication}
\mathcal{H}^m(R \cap E) >0 \ \Rightarrow \ \tilde{\mathscr{I}}^m_1(R)>0.
\end{equation}
We observe that, since both $\mathcal{H}^m$ and $\tilde{\mathscr{I}}^m_1$ are Borel regular measure, it is enough to prove implication \eqref{e:implication} assuming that $R$ is a Borel set. To this purpose, fix an $m$-plane $V$, and define the probability measure $\mu_V$ in $\mathbb{R}^n$ as
\begin{equation}
\label{e:defbesfed}
\mu_V(B):= \int_{\pi_V(E)} \frac{\#\big(E \cap B \cap \pi^{-1}_V(y)\big)}{\#\big(E\cap \pi^{-1}_V(y)\big)} \, d\mathcal{H}^m(y), \ \ \text{ for every Borel set $B \subset \mathbb{R}^n$},
\end{equation}
and possibly extended to all subsets of $\mathbb{R}^n$ as the unique Borel regular measure satisfying \eqref{e:defbesfed}. Then, we define $\tilde{\mathscr{I}}^m_1$ as the unique Borel regular measure in $\mathbb{R}^n$ satisfying 
\[
\tilde{\mathscr{I}}^m_1(B):= \int_{\text{Gr}(n,m)} \mu_V (B) \,d\gamma_{n,m}(V), \ \ \text{ for every Borel set $B \subset \mathbb{R}^n$}.
\]
The measurability required to define the above integrals can be found in \cite[Subsection 2.10.16]{fed1}. Now, given a counatbly $m$-rectifiable set $R \subset \mathbb{R}^n$, we apply Coarea formula with the projection $\pi_V$ to write 
\[
\int_{R \cap B} |\langle d\pi_v,\tau \rangle| \, d\mathcal{H}^m = \int_{V} \#\big( R \cap B \cap \pi_V^{-1}(y)  \big) \, d\mathcal{H}^m(y), \ \ \text{ for every Borel set $B \subset \mathbb{R}^n$}.
\]
It is not difficult to verify that, the above formula allows us to infer
\[
\mathcal{H}^m(R \cap E) >0 \ \Rightarrow \ \mu_V(R)>0, \ \ \text{ for $\gamma_{n,m}$-a.e. $V \in \text{Gr}(n,m)$}.
\]
This last implication in turn directly implies condition \eqref{e:implication}.

 By a standard saturation argument (see for instance \cite[Subsection 3.2.14]{fed1}), we can decompose the set $E$ as 
\[
E = R \cup (E \setminus R),
\]
where $R \subset \mathbb{R}^n$ is a countably $m$-rectifiable and Borel set, while $E \setminus R$ is purely $(\tilde{\mathscr{I}}^m_1,m)$-unrectifiable. First we observe that, condition \eqref{e:implication} tells us that the set $E \setminus R$ is also purely $(\mathcal{H}^m,m)$-unrectifiable. 

To conclude the proof, we need to verify property \eqref{e:spproperty}. For this purpose fix a finite measure $\mu$ in $\mathbb{R}^n$. By the previous claim, we can decompose
\[
\mu \restr (E \setminus R) = \mu_r + \mu_s.
\]
The desired property \eqref{e:spproperty} will follow if we show that $\mu_r = 0$. But since $\mu_r \restr (E \setminus R)$ must be an $m$-rectifiable measure, and $E \setminus R$ is purely $(\mathcal{H}^m,m)$-unrectifiable, this forces $\mu_r = 0$. Thanks to the arbitrariness of the measure $\mu$ this proves property \eqref{e:spproperty} and concludes the proof.
\end{proof}

As a direct consequence of the previous theorem we have the following result.

\begin{proof}[Proof of Corollary \ref{c:singlescale}]
    Let $\mu$ be the measure given by the Corollary. The condition \eqref{e:ssfav1} immediately implies that, in the decomposition of $K$ provided by Theorem~\ref{t:besfedintro}
    \[
    K = R \cup (K \setminus R),
    \]
     the Borel and countably $1$-rectifiable set $R$ must satisfy $\mathcal{H}^1(R) > 0$. Indeed, since every measure $\mu$ supported on an $\mathcal{H}^1$-negligible set satisfies property \eqref{e:ssfav1}, assuming by contradiction that $\mathcal{H}^1( R) = 0$ would immediately gives
     \[
     \pi_{\sharp \ell} \mu \restr K \, \bot \, \mathcal{H}^1, \ \ \text{ for $\gamma_{2,1}$-a.e. $\ell \in \text{Gr}(2,1)$},
     \]
     leading directly to a contradiction with condition \eqref{e:ssfav1}. The desired result thus follows from Calder\'on's theorem.
\end{proof}

\subsection{A multi-scale result} 
\label{ss:msresult}

Among the class of compact sets with finite integral geometric measure, the following multi-scale result constitute our main contribution in connection with Vitushkin's conjecture.

\begin{theorem}
\label{t:multiscalevit}
    Let $K \subset \mathbb{R}^n$ be a compact set with $\mathcal{I}^m_1(K) < \infty$. Assume that for every $x \in K$ there exist scales $r_{x,i} \searrow 0$ such that
    \begin{equation}
    \label{e:msfav1}
        \emph{Fav}(K \cap B_{r_{x,i}}(x)) >  k_x\, \emph{Fav}_{\Phi}(K \cap B_{r_{x,i}}(x)), \ \ \text{ for every $i=1,2,\dotsc$},
    \end{equation}
    for some constant $k_x$ that might depend on $x \in K$ and for some Orlicz function $\Phi$ with superlinear growth at infnity. Then, there exists a countably $m$-rectifiable Borel set $R \subset \mathbb{R}^n$ such that $\mathcal{H}^m(K \cap R)>0$.
\end{theorem}
\begin{proof}
    Appealing to Theorem \ref{t:cald}, we construct a nontrivial countably $m$-rectifiable Borel set $R \subset K$ via an integral geometric argument. The condition $\gamma(K) > 0$ will then follow thanks to Theorem \ref{t:cald}.
    
    First, note that the strict inequality in \eqref{e:msfav1} implies $\text{Fav}(K) > 0$, and thus $\mathcal{I}^m_1(K) > 0$. Therefore, since $\mathcal{I}^m_1(K) < \infty$, we can choose a compact subset $K' \subset K$ with $\mathcal{I}^m_1(K') > 0$ such that for every $x \in K'$
\begin{equation}
    \label{e:msfav1.1}
        \text{Fav}(K \cap B_{r_{x,i}}(x)) \geq c \, \text{Fav}_{\Phi}(K \cap B_{r_{x,i}}(x)), \ \ \text{ for every $i=1,2,\dotsc$},
    \end{equation}
    for some positive constant $c$ independent form $x \in K'$.
     By applying Besicovitch's covering theorem, for every $\delta > 0$ we can find a dimensional constant $N$ and a covering $\{\mathcal{B}^1_\delta, \dotsc, \mathcal{B}^N_\delta\}$ of $K'$ such that each family $\mathcal{B}^i_\delta$ consists of pairwise disjoint open balls, each satisfying \eqref{e:msfav1.1} and having diameter less than or equal to $\delta$. We can thus write
    \begin{align*}
    N \mathcal{I}^m_1(K) \geq \sum_{i=1}^N\sum_{B_r(x) \in \mathcal{B}^i_\delta}  \mathcal{I}^m_1(K \cap B_r(x)) &\geq \sum_{i=1}^N \sum_{B_r(x) \in \mathcal{B}^i_\delta}  \, \text{Fav}(K \cap B_r(x)) \\
    &\geq \sum_{i=1}^N \sum_{B_r(x) \in \mathcal{B}^i_\delta} c \, \text{Fav}_\Phi(K \cap B_r(x)) \\
    &\geq \inf_{\mathcal{G}_\delta} \sum_{B \in \mathcal{C}_\delta}  c \,\text{Fav}_\Phi(K \cap B) \\
    &\geq \inf_{\mathcal{G}_\delta} \sum_{B \in \mathcal{C}_\delta}  c \,\text{Fav}_\Phi(K' \cap B),
    \end{align*}
 where the infimum above is taken with respect to all the coverings $\mathcal{G}_\delta$ of $K'$ made of Borel set with diameter less than or equal to $\delta$. Thanks to the arbitrariness of $\delta$ we immediately deduce that $c\mathcal{I}^m_\Phi(K') \leq N \mathcal{I}^m_1(K) < \infty$. By Theorem \ref{t:fedpronew} and Remark \ref{r:intgeophi} this immediately implies the existence of a countably $m$-rectifiable Borel set $R \subset \mathbb{R}^n$ such that $\mathcal{I}^m_\Phi(K' \setminus R)=0$. Finally, since $\mathcal{I}^m_1(K') >0$ implies $\mathcal{I}^m_\Phi(K') >0$ we apply Proposition \ref{p:coihmim} to infer $\mathcal{H}^m(K \cap R) >0$. This gives the desired result. 
\end{proof}

As a direct consequence of the previous result we can give the proof of Theorem \ref{t:multiscaleintro} 

\begin{proof}[Proof of Theorem \ref{t:multiscaleintro}]
    By applying Theorem \ref{t:multiscalevit}, we obtain a countably $1$-rectifiable Borel set $R \subset \mathbb{C}$ such that $\mathcal{H}^1(K \cap R) > 0$. The desired conclusion then follows from Calderón's theorem. 
\end{proof}

We conclude this subsection by verifying the geometric meaning of condition \eqref{e:msfav1}, as anticipated in the introduction. Recall that \eqref{e:msfav1intro} is equivalent to:
\begin{equation} \label{e:invholder}
    \|\theta_K(\cdot,x,r)\|_\Phi \lesssim_x \int_{\mathrm{Gr}(2,1)}  \theta_K(\ell,x,r) \, d\gamma_{2,1}(\ell),
\end{equation}
where we define
\[
\theta_K(\ell,x,r) := \frac{\mathcal{H}^1\big(\pi_\ell(K \cap B_r(x))\big)}{r}, \quad \text{for } (\ell,x,r) \in \mathrm{Gr}(2,1) \times K \times (0,\infty).
\]
As already mentioned in the introduction, the interesting regimes, which are excluded by the ULFL condition, are precisely those in which $\theta_K(\ell,x,r) \to 0$ as $r \to 0^+$ for almost every lines $\ell \in \text{Gr}(2,1)$. A prototypical example of this behavior is the following asymptotics of the Favard profile at $x$:
\begin{equation}
\label{e:deftheta}
\theta_K(\ell,x,r) =
\begin{cases}
    1  & \text{if } \ell \in \Lambda_r \subset \mathrm{Gr}(2,1), \\
    c_r & \text{if } \ell \in \mathrm{Gr}(2,1) \setminus \Lambda_r,
\end{cases}
\end{equation}
for a family of Borel subsets $(\Lambda_r)_{r > 0}$ in $\mathrm{Gr}(2,1)$ and a family of constants $(c_r)_{r > 0}$ satisfying $c_r \to 0$ as $r \to 0$ and
\begin{equation}
\label{e:constraintms}
c_r = o(\gamma_{2,1}(\Lambda_r)), \quad \text{ as $r \to 0$ }.
\end{equation}
 
 We now verify that, under the constraint \eqref{e:constraintms}, condition~\eqref{e:invholder} holds. Indeed, using the formula

\[
\|\mathbbm{1}_{\Lambda}\|_\Phi = \frac{1}{\Phi^{-1}\big(\frac{1}{\gamma_{2,1}(\Lambda)}\big)},
\]
where $\mathbbm{1}_{\Lambda}$ denotes the characteristic function of the set $\Lambda$, it is not difficult to verify that condition \eqref{e:constraintms} allows one to build an Orlicz function $\Phi$ with superlinear growth at infinity such that
\begin{equation}
\|\mathbbm{1}_{\Lambda_r}\|_\Phi \leq c_r, \ \ \text{ for every sufficiently small $r > 0$ }.
\end{equation}
From this we easily estimate
\begin{align*}
    \|\theta_K(\cdot,x,r)\|_\Phi &\leq \| \mathbbm{1}_{\Lambda_r}\|_\Phi + c_r\|(1- \mathbbm{1}_{\Lambda_r})\|_\Phi
    \leq c_r(1 + \|(1- \mathbbm{1}_{\Lambda_r})\|_\Phi) \\
    &\leq \frac{1+\|(1- \mathbbm{1}_{\Lambda_r})\|_\Phi}{1-\gamma_{2,1}(\Lambda_r)} \int_{\text{Gr}(2,1) } \theta_K(\ell,x,r) \, d\gamma_{2,1}, 
\end{align*}
from which \eqref{e:invholder} follows.

\newpage

\appendix

\section{Ortoghonal projections as transversal maps}
\label{a:protra}
In this appendix we verify that the family of orthogonal projections $(\pi_V)_{V \in \text{Gr}(n,m)}$ can be represented as a finite union of transversal families in the sense of Definition \ref{d:transversal}. Notice that the same property immediately applies for the family $(\pi_{g(V)})_{g \in \text{O}(n)}$ for any fixed $V \in \text{Gr}(n,m)$. 

We equip $\text{Gr}(n,m)$ with the distance 
\begin{equation}
\label{e:metrgr}
d(V_1,V_2) := \sup_{\substack{x \in \mathbb{R}^n \\ |x|=1}} |\pi_{V_1}(x) -\pi_{V_2}(x)|.
\end{equation}
 and notice that $d(\cdot,\cdot)$ is invariant under the action of $\text{O}(n)$. For this reason, Hausdorff measures constructed with respect to this metric is invariant under such a action. As a consequence, the structure of $m(n-m)$-dimensional manifold of $\text{Gr}(n,m)$ allows us to identify $\gamma_{n,m}$ with a constant multiple of $\mathcal{H}^{m(n-m)}$. We further observe that, denoting $(V_i)$ the collection of all coordinate $m$-planes of $\mathbb{R}^n$ and defining
\[
U_i := \{V \in \text{Gr}(n,m) \ |\ \det (\pi_{V_i} \restr V) > (2\sqrt{c_{n,m}})^{-1} \},
\]
where $c_{n,m}$ is the cardinality of the family $(V_i)$, then the sets
$(U_i)$ form an open covering of $\text{Gr}(n,m)$. We observe that, in order to obtain the desired property, it is sufficient to find a finite open refinement of $(U_i)$, say $(W_j)$, open sets $(\Lambda_j)$ in $\mathbb{R}^{m(n-m)}$, and bi-Lipschitz diffeomorphisms $(\varphi_j)$ where $\varphi_j \colon \Lambda_j \to W_j$ such that
\begin{equation}
\label{e:proj2}
    |(\pi_{V_{i(j)}} \circ \pi_{x})(\varphi_j(\lambda))| \leq C' \ \ \text{implies} \ \ J_\lambda (\pi_{V_{i(j)}} \circ \pi_{x})(\varphi_j(\lambda)) \geq C', \ \ \lambda \in \Lambda_j, \ x \in \mathbb{S}^{n-1},
\end{equation}
for some constant $C'>0$, where $\pi_x \colon \text{Gr}(n,m) \to \mathbb{R}^n$ is defined as $\pi_x(V):= \pi_V(x)$ ($x \in \mathbb{R}^n$), and where $W_j \subset U_{i(j)}$ for every $j$. Indeed, defining $P_\lambda \colon \mathbb{R}^n \to \mathbb{R}^m$ as 
\begin{equation}
\label{e:proj3}
P^{j}_\lambda(x):= (\pi_{V_{i(j)}} \circ \pi_x)(\varphi_j(\lambda)), \ \ \lambda \in  \Lambda_j,
\end{equation}
by linearity we have 
\[
\frac{P^{j}_\lambda(x) - P^{j}_\lambda(x')}{|x-x'|}= P^{j}_\lambda\bigg(\frac{x-x'}{|x-x'|}\bigg),
\]
and hence, from \eqref{e:proj2}, the families $(P^{j}_\lambda)$ with $\lambda \in \Lambda_j$ satisfy hypothesis (H.2) of transversality. Since the remaining hypothesis (H.1) and (H.3) are trivially satisfied we deduce that $(P^{j}_\lambda)$ is transversal for every $j$. 

It remains to discuss the validity of \eqref{e:proj2}. First we check that 
\begin{equation}
\label{e:proj1}
|(\pi_{V_i} \circ \pi_{x})(V)| \leq C'' \ \ \text{implies} \ \ \text{Rank}\, (\pi_{V_i} \circ \pi_{x})(V) =m, \ \ V \in U_i, \ x \in \mathbb{S}^{n-1}
\end{equation}
for some $C''>0$. If not, we use the definition of $U_i$ together with a compactness argument to find $V_0 \in \overline{U}_i$ and $x \in V_0^\bot \cap \mathbb{S}^{n-1}$ such that
\begin{equation}
\label{e:contr}
\text{Rank}\, (\pi_{V_i} \circ \pi_{x})(V_0) <m.
\end{equation}
 By exploiting the (transitive) action of $\text{O}(n)$ on $\text{Gr}(n,m)$, condition \eqref{e:contr} implies that, defining $G \colon \text{O}(n) \to \mathbb{R}^m$ as $G(g):= (\pi_{V_i} \circ \pi_{x})(g \, V_0)$, then
\begin{equation}
\label{e:contr1}
\text{Rank}\, G(\text{Id}) <m,
\end{equation}
where Id denotes the identity $(n \times n)$ matrix. Using that the tangent space of $\text{O}(n)$ at the identity is isomorphic to $\mathbb{M}^{n \times n}_{skw}$, namely, the space of $(n \times n)$ skew symmetric matrices, it is not difficult to show that the differential of $G$ computed at Id has the following form
\begin{equation}
\label{e:differ}
-(\pi_{V_i} \circ \pi_{V_0})(A x), \ \ A \in \mathbb{M}^{n \times n}_{skw}.
\end{equation}
By denoting $\{e_1,\dotsc,e_m\}$ an orthonormal basis of $V_0$ we can consider its completion to an orthonormal basis of $\mathbb{R}^n$, say $\{e_1,\dotsc,e_n\}$, in such a way that $x=e_n$. Writing \eqref{e:differ} in this coordinates we get
\begin{equation}
\label{e:differ1}
    - [(Ae_n \cdot e_1)\pi_{V_i}(e_1)+ \dotsc + (Ae_n \cdot e_m)\pi_{V_i}(e_m)], \ \ A \in \mathbb{M}^{n \times n}_{skw}.
\end{equation}
Since $V_0 \in \overline{U}_i$ we have that $\pi_{V_i}(e_1), \dotsc, \pi_{V_i}(e_m)$ are linearly independent vectors of $V_i$. Hence the expression in \eqref{e:differ1} vanishes if and only if $(Ae_n \cdot e_i)=0$ for every $i=1,\dotsc,m$. But this means that the dimension of $\text{Im}(\pi_{V_i} \circ \pi_{V_0})(A x)$ for $A \in \mathbb{M}^{n \times n}_{skw}$ equals $m$. Finally, this is in contradiction with \eqref{e:contr} and we deduce the validity of \eqref{e:proj1}. From the compactness of $\text{Gr}(n,m)$ we easily find an open refinement of $(U_i)$, say $(W_j)$, open subsets of $ \mathbb{R}^{m(n-m)}$, say $(\Lambda_j)$, and bi-Lipschitz diffeomorphisms $(\varphi_j)$ such that $\varphi_j \colon \Lambda_j \to W_j$. To conclude, we combine \eqref{e:proj1} with the bi-lipschitzianity of $\varphi_j$ to infer the validity of \eqref{e:proj2}

\section{The Structure Theorem}
\label{a:structure}

In order to prove Proposition \ref{p:e1delta} we need the following lemma.

\begin{lemma}
\label{l:geometrical}
Let $E \subset \mathbb{R}^n$ be a purely $(\phi,m)$-unrectifiable set, let $0<s<1$, $\alpha >0$, $\delta >0$, $\lambda \in \Lambda$, and 
\begin{equation}
\label{e:lemma1.1}
    \phi(E \cap X(x,r,\lambda,s))\leq \alpha \omega(m) (rs)^m,
\end{equation}
whenever $x \in E$ and $0< r \leq \delta$, then
\begin{equation}
    \label{e:lemma1.2}
    \phi(E \cap \overline{B}_{4\rho/s}(a) \cap P_\lambda^{-1}\overline{B}_\rho(P_\lambda(a))) \leq 2(84)^m \alpha \omega(m) \rho^m, 
\end{equation}
whenever $a \in \mathbb{R}^n$ and $0 <\rho \leq s\delta/24$. In particular
\begin{equation}
    \label{e:lemma1.3}
    \Theta^{*m}(\phi \restr E, a) \leq 2(84)^m\alpha.
\end{equation}
\end{lemma}
\begin{proof}
We claim that given $E \subset \mathbb{R}^n$ satisfying
\begin{equation}
\label{e:conerect}
E \cap X(a,\infty,\lambda,s) = \emptyset \ \ \text{whenever } a \in E,
\end{equation}
then $E$ is countably $m$-rectifiable. To show this, notice that \eqref{e:conerect} implies for every $x,z \in E$  
\[
|P_\lambda(x)-P_\lambda(z)| \geq s|x-z|,
\]
hence $P_\lambda \restr E$ is injective and $f \colon \text{Im}(P_\lambda \restr E) \to \mathbb{R}^n$ satisfying $P_\lambda \restr E \circ f(y) = y$ is a well defined map with Lipschitz constant equal to $s^{-1}$. Now the proof follows straightforwardly the one of \cite[Lemma 3.3.6]{fed1}, where the orthogonal projection $p$ is replaced by $P_\lambda$.
\end{proof}

\begin{proof}[Proof of Proposition \ref{p:e1delta}]

 Consider first a set $E' \subset E$ with $\text{diam}(E') \leq \delta/6$. Given $\alpha >0$ and $a \in E_{1,\delta}(\lambda) \cap E'$ by hypothesis we know that there exists $0<\overline{s}_a \leq 1$ (depending also on $\alpha$) such that \eqref{e:lemma1.1} holds true whenever $0 < r \leq \delta$ and $0 < s \leq \overline{s}_a$. If we choose $\rho$ and $s$ such that $0 <\rho := s\delta/24 \leq \overline{s}\delta/24$, we can make use of \eqref{e:lemma1.2} to write 
\begin{equation}
    \label{e:prooflemma1.1}
    \phi(E' \cap P^{-1}_\lambda \overline{B}_{\rho}(P_\lambda(a))) \leq 2(84)^m \alpha \omega(m) \rho^m, 
\end{equation}
where we have used the fact that, thanks to the choice of $\rho$, $B \subset \overline{B}_{4\rho/s}(a) =\overline{B}_{\delta/6}(a)$ for every $a \in E'$ since $\text{diam}(E') \leq \delta/6$. This means that for every $y \in P_\lambda(E_{1,\delta}(\lambda) \cap E')$ 
\begin{equation}
    \label{e:prooflemma1.2}
   \limsup_{\rho \to 0^+} P_{\lambda\sharp} \phi \restr E'(\overline{B}_{\rho}(y))\rho^{-m} \leq 2(84)^m\alpha \omega(m).
\end{equation}
 We can apply the density estimates for Radon measures to infer
 \[
 P_{\lambda\sharp} \phi \restr E'(P_\lambda(E_{1,\delta}(\lambda) \cap E')) \leq 2^{2m+1} (84)^m\alpha \omega(m) \mathcal{L}^m(P_\lambda(E_{1,\delta}(\lambda) \cap E')).
 \]
 The arbitrariness of $\alpha$ implies $P_{\lambda\sharp} \phi \restr B(P_\lambda(E_{1,\delta}(\lambda) \cap E'))=0$. From the definition of pushforward (see formula \eqref{e:pushforward}) there exists $B$ Borel such that $P_\lambda(E_{1,\delta}(\lambda) \cap E') \subset B$ and $0=P_{\lambda\sharp}\phi \restr E'(B)= \phi(P^{-1}_\lambda(B) \cap E')$. Finally, since $E_{1,\delta}(\lambda) \cap E' \subset  P^{-1}_\lambda(B) \cap E'$ we conclude $\phi(E_{1,\delta}(\lambda) \cap E')=0$. In the general case we can simply write $E$ as the union of at most countably many sets $E_i$ $i=1,2,\dotsc$ such that $E_i \subset E$ and $\text{diam}(E_i) \leq \delta/6$ for $i=1,2,\dotsc$. The previous result applied to each $E_i$ tells us $\phi(E_{1,\delta}(\lambda) \cap E_i)=0$ which immediately gives the desired conclusion. 
\end{proof}

\vspace{3mm}

\begin{proof}[Proof of Proposition \ref{p:e2delta}]
By definition we have $P_\lambda(X(x,r,\lambda,s)) \subset \overline{B}_{rs}(P_\lambda(x))$ whenever $x \in \mathbb{R}^n$ and $r,s$ are positive numbers. Our hypothesis implies that, given $x \in E_{2,\delta}(\lambda)$, we can find sequences $(s_i)$ and $(r_i)$ with $s_i \to 0$ as $i \to \infty$ and $0< r_i < \delta$ for $i=1,2,\dotsc$, for which 
\[
\lim_{i \to \infty} P_{\lambda\sharp}\phi \restr E\big(\overline{B}_{r_i s_i}(P_\lambda(x))\big)(r_i s_i)^{-m}=\infty.
\]
 But this means that 
\[
\limsup_{\rho \to 0^+}P_{\lambda\sharp}\phi \restr E (\overline{B}_\rho(y))\rho^{-m}=\infty,
\]
whenever $y \in P_\lambda(E_{2,\delta}(\lambda))$. Since $\phi(E)<\infty$, by using \cite[Theorem 6.9]{mat1}, we finally infer $\mathcal{L}^m\big(P_\lambda(E_{2,\delta}(\lambda))\big)=0$.  
\end{proof}

\vspace{3mm}

Next we prove the three key alternatives of Proposition \ref{p:keycond}. The proof presented here is mainly inspired by the one in \cite[Lemma 2.5]{hov}.

\begin{proof}[Proof of Proposition \ref{p:keycond}]
 We can assume without loss of generality that $E$ is bounded. By exploiting the fact that $\phi \restr E$ is Radon a measure and hence inner regular, we find a $\sigma$-compact set $E'$ with $E' \subset E$ satisfying $\phi(E \setminus E')=0$. Notice that, if the conclusion of the proposition holds true with $E$ replaced by $E'$, then the same must be true for the original set $E$. We may thus suppose that $E$ is a $\sigma$-compact set. Fix $a \in \mathbb{R}^n$, $\lambda_0 \in \Lambda$ and $0 < \delta < \delta_0$ such that $\overline{B}_{2\delta_0}(\lambda_0) \subset \Lambda$. Let $V \subset \mathbb{R}^l$ be a $m$-dimensional linear subspace and let $V_{\sigma} := V + \sigma$ for all $\sigma \in  V^\bot$. For all $\sigma \in V^\bot$, define a measure $\psi_\sigma$ on $\Lambda$ and supported on $V_\sigma \cap \overline{B}_{\delta_0}(\lambda_0)$ by 
\begin{equation}
    \label{e:key1}
    \psi_\sigma(\Sigma) := \sup_{0<r<\delta} r^{-m} \phi((E \setminus \{a\}) \cap \overline{B}_r(a) \cap L_{V_{\sigma}}(\Sigma)),
\end{equation}
for all $\Sigma \subset \Lambda$, where
\begin{equation}
    \label{e:key2}
    L_{V_{\sigma}}(\Sigma) :=\! \! \! \! \! \! \! \! \! \! \! \bigcup_{\lambda \in \Sigma \cap V_\sigma \cap \overline{B}_{\delta_0}(\lambda_0)} \! \! \! \! \! \! \! \! P^{-1}_\lambda(P_\lambda(a)).
\end{equation}
We claim that the set 
\begin{equation}
    \label{e:key3}
    C_{V_{\sigma}} := \{\lambda \in \overline{B}_{\delta_0}(\lambda_0) \cap V_{\sigma} \ | \ (E \setminus \{a\}) \cap L_{V_{\sigma}}(\{\lambda\}) \cap \overline{B}_\delta(a) \neq \emptyset \}
\end{equation}
is $\mathcal{H}^m$-measurable. This follows from the fact that it is $\sigma$-compact which can be seen as follows. Define the $\sigma$-compact sets
\[
S_1 := \{(\lambda,x) \in (\overline{B}_{\delta_0}(\lambda_0) \cap V_{\sigma}) \times \mathbb{R}^n \ | \ P_\lambda(x) -P_\lambda(a) =0 \}
\]
and
\[
S_2 := S_1 \cap [\Lambda \times ((E \setminus \{a\}) \cap \overline{B}_\delta(a))],  
\]
then we have $C_{V_{\sigma}} = \pi_1(S_2)$, where $\pi_1 \colon \Lambda \times \mathbb{R}^n \to \Lambda$ denotes the orthogonal projection onto the first component. 

Let $D_{V_{\sigma}} := (\overline{B}_{\delta_0}(\lambda_0) \cap V_{\sigma}) \setminus C_{V_{\sigma}}$. From the definitions of $\psi_\sigma$ and $C_{V_{\sigma}}$ we deduce that $\psi_\sigma(D_{V_{\sigma}})=0$. Now \cite[Theorem 18.5]{mat1} implies that for $\mathcal{H}^m$-a.e. $\lambda \in \overline{B}_{\delta_0}(\lambda_0) \cap V_{\sigma}$ one of the following conditions holds true
\begin{align}
    \label{e:key4}
    &\limsup_{s \to 0^+} \psi_\sigma( \overline{B}_s(\lambda))s^{-m}=0, \\
    \label{e:key5}
    &\limsup_{s \to 0^+} \psi_\sigma( \overline{B}_s(\lambda))s^{-m}=\infty, \\
    \label{e:key6}
    &\ \ \ \ \ \ \ \ \ \ \ \ \   \lambda \in C_{V_{\sigma}}. 
\end{align}
We have thus proved that for every $\sigma \in V^\bot$ and for $\mathcal{H}^m$-a.e. $\lambda \in \overline{B}_{\delta_0}(\lambda_0) \cap V_\sigma$ one among \eqref{e:key4}-\eqref{e:key6} holds true. Denote by $\pi_{V^\bot} \colon \mathbb{R}^l \to V^\bot$ the orthogonal projection of $\mathbb{R}^l$ onto $V^\bot$. Notice that once we prove the $\mathcal{L}^l$-measurability of the set of points $\lambda \in \overline{B}_{\delta_0}(\lambda_0)$ for which one between \eqref{e:key4}-\eqref{e:key6} holds true with $\sigma$ replaced by $\pi_{V^\bot}(\lambda)$, then by applying Tonelli's theorem we see that for $\mathcal{L}^l$-a.e. $\lambda \in \overline{B}_{\delta_0}(\lambda_0)$ one of the following conditions holds true
\begin{align}
    \label{e:key7}
    &\limsup_{s \to 0^+} \psi_{\pi_{V^\bot}(\lambda)}( \overline{B}_s(\lambda))s^{-m}=0, \\
    \label{e:key8}
    &\limsup_{s \to 0^+} \psi_{\pi_{V^\bot}(\lambda)}( \overline{B}_s(\lambda))s^{-m}=\infty, \\
    \label{e:key9}
    &\ \ \ \ \ \ \ \ \ \ \ \ \    \lambda \in C_{V_{\pi_{V^\bot}(\lambda)}}. 
\end{align}
However, notice that the exceptional set of $\mathcal{L}^l$-measure zero depends on the $m$-plane $V$ and this fact does not allow to deduce \eqref{e:cond1}-\eqref{e:cond3} immediately from \eqref{e:key7}-\eqref{e:key9}. Nevertheless this issue can be solved by combining Proposition \ref{p:equivcone} with properties \eqref{e:key7}-\eqref{e:key9}. This allows us to infer the validity of \eqref{e:cond1}-\eqref{e:cond3} on $\overline{B}_{\delta_0}(\lambda_0)$ and thus the proposition follows by a simple covering argument. 

It remains to prove the measurability. In order to simplify the notation let us denote $\pi_{V^\bot}(\lambda)$ by $\sigma_\lambda$. First we prove that the map $\varphi_{a,r,s} \colon \overline{B}_{\delta_0}(\lambda_0) \to [0,\infty]$ defined by 
\begin{equation}
    \label{e:key10}
    \varphi_{a,r,s}(\lambda):=\phi\big(E \cap \overline{B}_r(a) \cap L_{V_{\sigma_\lambda}}(\overline{B}_s(\lambda))\big)
\end{equation}
is $\mathcal{L}^l$-measurable. For every $a \in E$, $r >0$, and $0<s<1$ define the $\sigma$-compact sets 
\begin{equation*}
    S_1(a,r,s) :=\{(\lambda',\lambda,x) \in  \Lambda^2 \times \mathbb{R}^n \ | \  P_{\lambda'}(x)= P_{\lambda'}(a),
    \ |\lambda' -\lambda|\leq s, \ \pi_{V^\bot}(\lambda' -\lambda)=0\},
\end{equation*}
\begin{equation*}
    S_2(a,r,s):= S_1(a,r,s) \cap [\overline{B}_{\delta_0}(\lambda_0) \times \overline{B}_{\delta_0}(\lambda_0) \times ((E\setminus \{a\}) \cap \overline{B}_r(a))].
\end{equation*}
Notice that 
\[
E \cap \overline{B}_r(a) \cap L_{V_{\sigma_\lambda}}( \overline{B}_s(\lambda)) = \pi_{2,3}(S_2(a,r,s))_\lambda,
\]
for every $\lambda \in \overline{B}_{\delta_0}(\lambda_0)$ where $\pi_{2,3} \colon \mathbb{R}^l \times \mathbb{R}^l \times \mathbb{R}^n \to \mathbb{R}^l \times \mathbb{R}^n$ denotes the orthogonal projection onto the second and third components. Moreover $S_2(a,r,s)$ is $\sigma$-compact because $S_1(a,r,s)$ is closed and hence $S_2(a,r,s)$ is the intersection of a closed set and a $\sigma$-compact set. This means that also $\pi_{2,3}(S_2(a,r,s))$ is $\sigma$-compact and we can apply Tonelli's theorem \cite[Proposition 5.2.1 ]{coh} on the product space $\Lambda \times \mathbb{R}^n$ with product measure $\mathcal{L}^{l} \otimes \phi \restr E$ to deduce that $\varphi_{a,r,s}(\cdot)$ is a $\mathcal{L}^l$-measurable map. For every $j=1,2,\dotsc$ notice that
\[
\begin{split}
\chi_{a,j}(\lambda)&:= \sup_{s < 1/j} \sup_{0<r<\delta} \varphi_{a,r,s}(\lambda)(rs)^{-m}\\
&= \sup_{s \in (0,1/j) \cap \mathbb{Q}} \ \sup_{s \in (0,\delta) \cap \mathbb{Q}} \varphi_{a,r,s}(\lambda)(rs)^{-m},
\end{split}
\]
whenever $\lambda \in \overline{B}_{\delta_0}(\lambda_0)$. Indeed we claim that for every $a \in E$ we have $\lim_{k \to \infty} \varphi_{a,r_k,s_k}(\lambda)= \varphi_{a,r,s}(\lambda)$ whenever $s_k \searrow s$ and $r_k \searrow r$ for every $0<s<1$ and $r>0$. This last fact is a consequence of the following pointwise convergence
\begin{equation}
\label{e:lebesgue}
\mathbbm{1}_{\overline{B}_{r_k}(a) \cap L_{V_{\sigma_\lambda}}(\overline{B}_{s_k}(\lambda))}(x) \to \mathbbm{1}_{\overline{B}_{r}(a) \cap L_{V_{\sigma_\lambda}}(\overline{B}_{s}(\lambda))}(x)
\end{equation}
\emph{for every} $x \in \mathbb{R}^n$ whenever $s_k \searrow s$ and $r_k \searrow r$. Our claim follows thus from Lebesgue's dominated convergence theorem applied to the sequence of Borel functions in \eqref{e:lebesgue} and with measure $\phi \restr E$. 

But this means that $\chi_{a,j}(\cdot)$ is a $\mathcal{L}^l$-measurable map. Finally, we notice that the points $\lambda \in \overline{B}_{\delta_0}(\lambda_0)$ for which \eqref{e:key7}-\eqref{e:key8} hold true are exactly 
\[
\{\lambda \in \overline{B}_{\delta_0}(\lambda_0) \ | \ \lim_{j \to \infty} \chi_{a,j}(\lambda) =0\}
 \ \
\text{ and } \ \ \{\lambda \in \overline{B}_{\delta_0}(\lambda_0) \ | \ \lim_{j \to \infty} \chi_{a,j}(\lambda) =\infty\},
\]
respectively. It remains therefore to prove the measurability of points satisfying \eqref{e:key9}. For this purpose define the $\sigma$-compact sets
\[
S_1 := \{(\lambda,x) \in \overline{B}_{\delta_0}(\lambda_0) \times \mathbb{R}^n \ | \ P_\lambda(x) -P_\lambda(a) =0 \}
\]
and
\[
S_2 := S_1 \cap (\Lambda \times [(E \setminus \{a\}) \cap \overline{B}_\delta(a)]),  
\]
then $C_{V_{\sigma_\lambda}} = \pi_1(S_2)$ where $\pi_1 \colon \mathbb{R}^l \times \mathbb{R}^n \to \mathbb{R}^l$ denotes the orthogonal projection onto the first component. Therefore $C_{V_{\sigma_\lambda}}$ is $\sigma$-compact and this concludes the proof.
\end{proof}

\section{Proofs of Proposition \ref{p:fubini} }
\label{a:fubini}
\begin{proof}[Proof of Proposition \ref{p:fubini}]
 Let $\mathcal{F}$ denote the collection of all sets $A \subset \mathbb{R}^n \times \Lambda$ for which the map
\[
\lambda \mapsto \mu_\lambda(A_\lambda)
\]
is $\mathcal{L}^l$-measurable. For $A \in \mathcal{F}$, we define
\[
\rho(A) := \int_\Lambda \mu_\lambda(A_\lambda) \, d\lambda.
\]

Define:
\[
\mathcal{P}_0 := \{B \times U \ | \  B \subset \mathbb{R}^n \text{ Borel}, \, U \subset \Lambda \text{ Borel} \},
\]
\[
\mathcal{P}_1 := \big\{ \bigcup_{j=1}^\infty A_j \ | \  A_j \in \mathcal{P}_0 \big\},
\quad
\mathcal{P}_2 := \big\{ \bigcap_{j=1}^\infty A_j \ | \  A_j \in \mathcal{P}_1 \big\}.
\]

Note that $\mathcal{P}_0 \subset \mathcal{F}$ and for $B \times U \in \mathcal{P}_0$, we have:
\[
\rho(B \times U) = \int_U \mu_\lambda(B) \, d\lambda.
\]

Also, $\mathcal{P}_0$ is closed under finite intersections and differences. Indeed, if $B_1 \times U_1, B_2 \times U_2 \in \mathcal{P}_0$, then:
\[
(B_1 \times U_1) \cap (B_2 \times U_2) = (B_1 \cap B_2) \times (U_1 \cap U_2) \in \mathcal{P}_0,
\]
\[
(B_1 \times U_1) \setminus (B_2 \times U_2) = [(B_1 \setminus B_2) \times U_1] \cup [(B_1 \cap B_2) \times (U_1 \setminus U_2)],
\]
which is a disjoint union of sets in $\mathcal{P}_0$. It follows that each element of $\mathcal{P}_1$ is a countable disjoint union of elements of $\mathcal{P}_0$. From this, it is not difficult to see that $\mathcal{P}_1 \subset \mathcal{F}$. We further notice that any finite intersections of elements of $\mathcal{P}_1$ can be written as the countable union of pairwise disjoint elements in $\mathcal{P}_0$. Hence, we have also that $A_1 \cap \dotsc \dotsc \cap A_N \in \mathcal{F}$ whenever $\{A_1 \dotsc, A_N\} \subset \mathcal{P}_1$ and $N$ is any positive integer.

\vspace{5mm}

We claim that for each $A \subset \mathbb{R}^n \times \Lambda$,
\begin{equation}
\label{e:infimizing}
\hat{\mathscr{I}}_m(A) = \inf \{ \rho(F) \ | \ A \subset F, \, F \in \mathcal{P}_1 \}.
\end{equation}

Let us prove the claim. If $A \subset F = \bigcup_{i=1}^\infty (B_i \times U_i)$ with $B_i \times U_i \in \mathcal{P}_0$, then:
\[
\rho(F) \leq \sum_{i=1}^\infty \rho(B_i \times U_i) = \sum_{i=1}^\infty \int_{U_i} \mu_\lambda(B_i) \, d\lambda,
\]
so by construction \eqref{e:fubinicon}
\[
\inf \{ \rho(F) \ | \ A \subset F \in \mathcal{P}_1 \} \leq \hat{\mathscr{I}}_m(A).
\]

Moreover, for any given $F \in \mathcal{P}_1$ there exists a disjoint sequence $(B'_i \times U'_i)_i \subset \mathcal{P}_0$ such that $F = \bigcup_i (B'_i \times U'_i) \in \mathcal{P}_1$. Hence, for any given $F \in \mathcal{P}_1$ with $A \subset F$ we can also write
\[
\rho(F) = \sum_i \int_{U'_i} \mu_\lambda(B'_i) \, d\lambda \geq \hat{\mathscr{I}}_m(A),
\]
completing the proof of \eqref{e:infimizing}.

\vspace{5mm}

Let us prove condition (2).
Fix $B \times U \in \mathcal{P}_0$. Then for all $F \in \mathcal{P}_1$ with $B \times U \subset F$,
\[
\hat{\mathscr{I}}_m(B \times U) \leq \int_U \mu_\lambda(B) \,d\lambda = \rho(B \times U) \leq \rho(F),
\]
so, by infimizing on all $F \in \mathcal{P}_1$ with $A \subset F$, Claim 1 implies:
\[
\hat{\mathscr{I}}_m(B \times U) = \int_U \mu_\lambda(B) \, d\lambda.
\]
This gives condition (2).

\vspace{5mm}

 Next we prove that $\hat{\mathscr{I}}_m$ is a Borel measure. First, we claim that any set $B \times U \in \mathcal{P}_0$ is $\hat{\mathscr{I}}_m$-measurable. So suppose $A \subset \mathbb{R}^n \times U$ and $A \subset F$ for some $F \in \mathcal{P}_1$. Then both $F \setminus (B \times U)$ and $F \cap (B \times U)$ are disjoint sets in $\mathcal{P}_1$. Consequently, since by construction we have that $F_\lambda$ is a Borel set in $\mathbb{R}^n$ for every $F \in \mathcal{P}_1$ and for every $\lambda \in \Lambda$, we can write
\begin{align*}
\hat{\mathscr{I}}_m(A \setminus (B \times U)) + \hat{\mathscr{I}}_m(A \cap (B \times U))
&\leq \rho(F \setminus (B \times U)) + \rho(F \cap (B \times U)) \\
&= \rho(F),
\end{align*}
where, in the last equality we used that, by assumption, $\mu_\lambda$ is a Borel measure for $\mathcal{L}^l$-a.e. $\lambda \in \Lambda$.
Thus, by Claim~\#1,
\[
\hat{\mathscr{I}}_m(A \setminus (B \times U)) + \hat{\mathscr{I}}_m(A \cap (B \times U)) \leq \hat{\mathscr{I}}_m(A).
\]
It follows that the set $B \times U$ is $\hat{\mathscr{I}}_m$-measurable and the claim is proved. Eventually, since the class $\mathcal{P}_0$ generates the Borel $\sigma$-algebra on $\mathbb{R}^n \times \Lambda$, we conclude that all Borel sets are $\hat{\mathscr{I}}_m$-measurable.

\vspace{5mm}

We claim that for every $A \subset \mathbb{R}^n \times \Lambda$, there exists $F \in \mathcal{P}_2 \cap \mathcal{F}$ such that $A \subset F$
\begin{equation}
\label{e:fubinieq1.3}
\hat{\mathscr{I}}_m(F)=\rho(F) = \hat{\mathscr{I}}_m(A).
\end{equation}
Notice that, since we already proved that $\hat{\mathscr{I}}_m$ is a Borel measure, once that the claim is proved, condition (1) immediately follows.

Let us prove the claim. If $\hat{\mathscr{I}}_m(A) = \infty$, let $F = \mathbb{R}^n \times \Lambda$. Then, clearly $\hat{\mathscr{I}}_m(\mathbb{R}^n \times \Lambda)= \infty$ and by \eqref{e:infimizing} we deduce also $\rho(\mathbb{R}^n \times \Lambda)=\infty$. Otherwise, by applying again \eqref{e:infimizing} for each $j$, there exists $F_j \in \mathcal{P}_1$ with $A \subset F_j$ and:
\begin{equation}
\label{e:idk1}
\rho(F_j) < \hat{\mathscr{I}}_m(A) + \frac{1}{j}.
\end{equation}
Set:
\begin{equation}
\label{e:idk2}
F := \bigcap_{j=1}^\infty F_j \in \mathcal{P}_2.
\end{equation}
Then, since $\rho(F_j) < \infty$, we know that $\mu_\lambda((F_j)_\lambda) < \infty$ for $\mathcal{L}^l$-a.e. $\lambda \in \Lambda$. Therefore, we can make use of the monotonicity property of measures, to infer that
\[
\mu_\lambda(F_\lambda) = \lim_{j \to \infty} \mu_\lambda \bigg( \bigcap_{k=1}^j (F_k)_\lambda \bigg), \ \ \text{ for $\mathcal{L}^l$-a.e. $\lambda \in \Lambda$},
\]
Hence, we deduce that $F \in \mathcal{F}$. Moreover, by Fatou's lemma:
\begin{align*}
\hat{\mathscr{I}}_m(A) \leq \rho(F) = \int_\Lambda \mu_\lambda(F_\lambda) \, d\lambda &= \int_\Lambda \lim_{j \to \infty} \mu_\lambda\bigg( \bigcap_{k=1}^j (F_k)_\lambda \bigg) \, d\lambda \\
&\leq \liminf_{j \to \infty} \int_\Lambda  \mu_\lambda\bigg( \bigcap_{k=1}^j (F_k)_\lambda \bigg) \, d\lambda \\
&\leq \liminf_{j \to \infty} \rho(F_j) \\
&\leq \hat{\mathscr{I}}_m(A),
\end{align*}
where we used also that finite intersection of elements of $\mathcal{P}_1$ belongs to $\mathcal{F}$. Thus, we have shown that $\hat{\mathscr{I}}_m(A) = \rho(F)$. To complete the proof of \eqref{e:fubinieq1.3}, we now assume further that $A \in \mathcal{P}_2$, so that $A = \bigcap_j F'_j$ for some sequence $(F'_j) \subset \mathcal{P}_1$. In this case, we modify the above construction by defining
\[
F := \bigcap_{j=1}^\infty (F_j \cap F'_j).
\]
Setting $F''_j := F_j \cap F'_j$, we note that each $F''_j \in \mathcal{F}$ satisfies $A \subset F''_j$ and still fulfills condition \eqref{e:idk1}. Moreover, by construction, $F = A$. Therefore, applying the same argument as before yields $\rho(F) = \hat{\mathscr{I}}_m(F)$, which proves \eqref{e:fubinieq1.3}.

\vspace{5mm}

Let us prove condition (3). Take a Borel set $A \subset \mathbb{R}^n \times \Lambda$. If $\hat{\mathscr{I}}_m(A) = 0$, we know that there exists $F \in \mathcal{P}_2 \cap \mathcal{F}$ such that $A \subset F$ and $\rho(F) = 0$. Then 
\[
\mu_\lambda(A_\lambda) \leq \mu_\lambda(F_\lambda) =0, \ \ \text{ for $\mathcal{L}^l$-a.e. $\lambda \in \Lambda$}.
\]
Hence $A \in \mathcal{F}$ and $\rho(A) = 0$.

Now suppose $\hat{\mathscr{I}}_m(A)< \infty$. Then there exists $F \in \mathcal{P}_2 \cap \mathcal{F}$ such that $A \subset F$ and $\hat{\mathscr{I}}_m(A) = \hat{\mathscr{I}}_m(F)$, and, being both sets $A$ and $F$ Borel, we have also 
\[
\hat{\mathscr{I}}_m(F \setminus A) = 0.
\]
Thus, $F \setminus A \in \mathcal{F}$ and $\rho(F \setminus A)=0$. Hence
\[
0=\int_\Lambda \mu_\lambda\big((F \setminus A)_\lambda\big) \, d\lambda.
\]

Finally, we observe that, since $F \in \mathcal{P}_2$, then it is not difficult to verify that $F_\lambda$ is a Borel set of $\mathbb{R}^n$ for every $\lambda \in \Lambda$. Therefore the set $A_\lambda = F_\lambda \setminus (F \setminus A)_\lambda$ is $\mu_\lambda$-measurable  as it is the difference between a Borel set and $\mu_\lambda$-negligible set (hence the difference between two $\mu_\lambda$-measurable sets). As a consequence we can infer
\[
0=\int_\Lambda \mu_\lambda\big((F \setminus A)_\lambda\big) \, d\lambda = \int_\Lambda |\mu_\lambda(F_\lambda) - \mu_\lambda (A_\lambda)| \, d\lambda,
\]
which immediately implies the $\mathcal{L}^\ell$-measurablity of the map $\lambda \mapsto \mu_\lambda(A_\lambda)$ as well as equality
\[
\hat{\mathscr{I}}_m(A) = \rho(F) = \int_{\Lambda} \mu_\lambda(F_\lambda)\, d\lambda = \int_{\Lambda} \mu_\lambda(A_\lambda)\, d\lambda.
\]
This concludes the proof of (3).

\end{proof}

\end{document}